\documentclass[11pt]{article}

\usepackage[margin=1.19in]{geometry}
\usepackage{amsmath,amssymb,amsthm}
\usepackage{booktabs,tabularx}
\usepackage{enumitem}
\usepackage{graphicx}
\usepackage{microtype}
\usepackage{xcolor}
\usepackage[hidelinks]{hyperref}
\usepackage{tikz}
\usetikzlibrary{arrows.meta,calc,decorations.pathreplacing,positioning}

\newcolumntype{Y}{>{\raggedright\arraybackslash}X}

\setlist[enumerate,1]{label=(\roman*)}

\newtheorem{theorem}{Theorem}[section]
\newtheorem{lemma}[theorem]{Lemma}
\newtheorem{corollary}[theorem]{Corollary}
\theoremstyle{definition}
\newtheorem{definition}[theorem]{Definition}
\newtheorem{proposition}[theorem]{Proposition}
\newtheorem{remark}[theorem]{Remark}
\newtheorem{claim}[theorem]{Claim}

\newenvironment{claimproof}
  {\begin{proof}}
  {\end{proof}}

\newcommand{\dmin}{\delta^0}
\newcommand{\RN}{RN^+}
\newcommand{\cC}{\mathcal C}

\definecolor{reviewblue}{RGB}{22,76,145}

\title{Inheritance of expansion and oriented Hamilton cycles in robust expanders: a regularity-free proof}
\author{Louis DeBiasio\thanks{Department of Mathematics, Miami University,
Oxford, OH, USA. E-mail: \texttt{debiasld@miamioh.edu}. Research supported
in part by AMS--Simons Research Enhancement Grants
for PUI Faculty GR001015.}}
\date{\today}

\begin{document}
\maketitle

\begin{abstract}
H\"aggkvist and Thomason proved that every $n$-vertex oriented graph with minimum
semidegree at least $(\frac{5}{12}+o(1))n$ contains every orientation of a Hamilton
cycle. Using Szemer\'edi's regularity lemma, Kelly later improved this to
the asymptotically sharp bound $(\frac{3}{8}+o(1))n$. Taylor subsequently
generalized Kelly's theorem by proving that every sufficiently large robust outexpander
with linear minimum semidegree contains every orientation of a Hamilton
cycle.  We revisit H\"aggkvist and Thomason's proof and show that it can be
modified to give a regularity-free proof of Taylor's theorem (and thus
Kelly's theorem).  Building on this, we also give regularity-free proofs of related results
on Hamilton-connectivity and linkage in robust outexpanders. As applications
of these results, we are able to replace the use of the regularity lemma in known
results on arbitrary orientations of Hamilton cycles in $n$-vertex digraphs
with minimum semidegree at least $\frac{n}{2}$ and in $n$-vertex digraphs with
minimum total degree at least $(1+o(1))n$.

The first step is to show that, in an $n$-vertex robust outexpander with linear minimum semidegree, two uniformly chosen disjoint sets of logarithmic size satisfy Hall's condition with high probability.
The second is to show that robust expansion is
inherited by uniformly chosen linear sized sets. In both cases, we use the
graph-container methods of Kleitman--Winston and Sapozhenko to reduce the possible obstructions to a small enough family
to permit a union bound.
\end{abstract}

\section{Introduction}

H\"aggkvist and Thomason~\cite{HT} proved that every oriented graph with
minimum semidegree at least $(\frac{5}{12}+o(1))n$ contains every orientation of a
Hamilton cycle and they conjectured that the correct threshold should be
$(\frac{3}{8}+o(1))n$. Kelly, K\"uhn, and Osthus~\cite{KellyKuhnOsthus} first proved
the corresponding approximate threshold for directed Hamilton cycles (using Szemer\'edi's regularity lemma~\cite{Szemeredi} together with a result of H\"aggkvist on 1-factors in oriented graphs), and
Keevash, K\"uhn, and Osthus~\cite{KKO} subsequently determined the exact
threshold. Kelly~\cite{Kelly} then used Szemer\'edi's regularity lemma together with the result of
Keevash, K\"uhn, and Osthus to prove H\"aggkvist and Thomason's conjecture. We first recall the definition of robust outexpansion. Let $0<\nu\leq\tau<1$, and let $G$ be a digraph on $n$
vertices. For
$X\subseteq V(G)$, its \emph{$\nu$-robust outneighborhood} is
$\RN_\nu(X)=\{v\in V(G):|N^-(v)\cap X|\geq\nu n\}$.
The digraph $G$ is a \emph{robust $(\nu,\tau)$-outexpander} if, for all
$X\subseteq V(G)$ with $\tau n\leq |X|\leq(1-\tau)n$, we have
$|\RN_\nu(X)|\geq |X|+\nu n$.

This definition was introduced implicitly by H\"aggkvist and
Thomason~\cite{HT}, who proved that every $n$-vertex oriented graph with
minimum semidegree at least $(\frac{3}{8}+o(1))n$ is a robust outexpander. K\"uhn,
Osthus, and Treglown~\cite{KOT} began the systematic study of robust
outexpansion. Using the regularity lemma, they proved that every
sufficiently large robust $(\nu,\tau)$-outexpander with linear minimum
semidegree contains a directed Hamilton cycle. Taylor~\cite{Taylor} then
used their result, together with another application of the regularity
lemma, to prove that the same hypotheses guarantee every orientation of a
Hamilton cycle. Later, Lo and Patel~\cite{LP} gave a regularity-free proof
of the K\"uhn--Osthus--Treglown theorem. This also gives regularity-free
proofs of the results of Keevash, K\"uhn, and Osthus and Kelly, K\"uhn, and
Osthus. This left the problem of proving Taylor's theorem without the
regularity lemma. We solve this problem here.

In this paper, we revisit H\"aggkvist and Thomason's proof and strengthen a
key step using the Kleitman--Winston algorithm~\cite{KW}.
First we need the following definition from H\"aggkvist and Thomason. Let $G$ be a digraph, let $\lambda>0$, and let $S_1,S_2\subseteq V(G)$ be
disjoint sets with $|S_1|=|S_2|$. The pair $(S_1,S_2)$ is
\emph{$\lambda$-expanding} if, for all $i\in [2]$ and for all $A\subseteq S_i$ with $0<|A|\leq \frac{|S_i|}{2}$ we have $|N^\pm(A)\cap S_{3-i}|\geq(1+\lambda)|A|$.

Now we state our two main theorems. Neither statement follows directly
from the Chernoff--Hoeffding bound because there are too many possible
obstructions for a direct union bound. Under the stronger hypothesis
$\delta^0(G)\geq(\frac{5}{12}+o(1))n$, H\"aggkvist and Thomason obtained the following Hall-type conclusion; see
\cite[Lemmas~8 and~10]{HT}.

\begin{theorem}[Hall-type inheritance]\label{thm:hall-inherit}
Let $0<\nu\leq\tau\leq\frac{\eta}{8}<\frac{1}{8}$ and $n\geq\frac{16}{\nu^2}$.
If $G$ is an $n$-vertex robust
$(\nu,\tau)$-outexpander with $\dmin(G)\geq\eta n$, and $(S,S')$ is
chosen uniformly from the ordered pairs of disjoint $m$-subsets of $V(G)$,
where $1\leq m\leq \frac{n}{2}$, then the probability that $(S,S')$ is not
$\frac{\nu}{4}$-expanding is at most
$8n^{48/\nu^2}\exp(-\frac{\nu^2m}{288})$.
\end{theorem}

In particular, for all $\lambda>0$, the pair $(S,S')$ in
Theorem~\ref{thm:hall-inherit} is $\frac{\nu}{4}$-expanding with probability at
least $1-n^{-\lambda}$ whenever
$m\geq\frac{288}{\nu^2}((\frac{48}{\nu^2}+\lambda)\log n+\log 8)$.

The following is an explicitly quantified, regularity-free version of a
robust-inheritance statement implicit in Taylor's proof. Gir\~ao, Granet,
K\"uhn, Lo, and Osthus~\cite[Lemma~14.3]{GGKLO} later stated a slightly more
general version for random sets of linear size. Their result follows from
the same regularity-based argument; a full proof appears in Granet's
thesis~\cite[Lemma~B.5]{Granet}. We give a regularity-free proof of this
result while extending inheritance of robust expansion to random sets of
arbitrary growing size.

\begin{theorem}[Robust inheritance]
\label{thm:robust-inherit}
There exists an absolute constant $K_1>0$ such that the following holds.
Let $0<\nu\leq\tau\leq\frac{\eta}{2}<\frac{1}{2}$, and set
$\rho=\frac{\nu^8}{K_1\log^2(K_1/\nu)}$.
For all $n\geq2$, let $G$ be an $n$-vertex robust
$(\nu,\tau)$-outexpander satisfying $\dmin(G)\geq\eta n$, and let $W$ be
a uniformly random $m$-subset of $V(G)$, where
$\frac{K_1}{\nu^4}\log(\frac{K_1}{\nu})\leq m\leq n$. The following estimates hold:
\begin{align*}
 \Pr\bigl(G[W]\text{ is not a robust $(\rho,2\tau)$-outexpander}\bigr)
 &\leq2\exp\left(-\frac{\nu^2m}{64}\right),\\
 \Pr\left(\exists v\in V(G),~ \sigma\in\{+,-\}:
 |N^\sigma(v)\cap W|<\frac{\eta m}{2}\right)
 &\leq2n\exp\left(-\frac{\eta^2m}{2}\right).
\end{align*}
\end{theorem}

We make no attempt to optimize the absolute constant $K_1$.
In particular, for all $\lambda>0$, both conclusions of
Theorem~\ref{thm:robust-inherit} hold simultaneously with probability at
least $1-n^{-\lambda}$ whenever
\begin{equation*}
 m\geq\max\left\{
 \frac{K_1}{\nu^4}\log(\frac{K_1}{\nu}),
 \frac{64}{\nu^2}(\lambda\log n+\log 4),
 \frac{2}{\eta^2}\bigl((\lambda+1)\log n+\log 4\bigr)
 \right\}.
\end{equation*}

\subsection{Applications}

Theorem~\ref{thm:hall-inherit} supplies the missing ingredient that allows one to generalize the
H\"aggkvist--Thomason argument to any robust
outexpander of linear minimum semidegree. Thus we give a regularity-free proof of
Taylor's theorem in the following stronger form.

A digraph $G$ is \emph{universally Hamilton-connected} if, for all distinct
$x,y\in V(G)$ and all oriented paths $P$ on $V(G)$, it contains a copy
of $P$ whose first and last vertices are $x$ and $y$, respectively.
Here first and last refer to the underlying order of $P$, independently
of the directions of its end edges.

\begin{theorem}\label{thm:main}
There exists an absolute constant $K_2>0$ such that the following holds. Let
$0<\nu\leq\tau\leq\frac{\eta}{8}<\frac{1}{8}$ and
$n\geq\exp(\frac{K_2}{\nu^4})$, and let $G$ be an $n$-vertex digraph.
If $G$ is a robust $(\nu,\tau)$-outexpander and satisfies
$\delta^0(G)\geq\eta n$, then $G$ is universally Hamilton-connected.
In particular, $G$ contains every orientation of a Hamilton cycle.
\end{theorem}

The proof of DeBiasio, K\"uhn, Molla, Osthus, and Taylor~\cite{DKMOT}
uses Taylor's theorem in the non-extremal case; their extremal argument
does not use the regularity lemma\footnote{The extremal arguments in~\cite{DKMOT} require only polynomial lower bounds on $n$ in terms of their auxiliary parameters, while the
robust-expander case now uses $n\geq\exp(\frac{K_2}{\nu^4})$.
We do not attempt to optimize the resulting absolute constant $n_0$.}. By replacing the use of Taylor's
theorem with Theorem~\ref{thm:main} and retaining the remainder of their
proof unchanged, we obtain a regularity-free proof of the following result.

\begin{corollary}[\cite{DKMOT}]\label{thm:DKMOT}
There exists a constant $n_0$ such that if $G$ is an $n$-vertex digraph with $n\geq n_0$ and $\delta^0(G)\geq \frac{n}{2}$, then $G$ contains every orientation of a Hamilton cycle (except perhaps the anti-directed orientation).
\end{corollary}

The anti-directed case was already treated without the use of the regularity lemma by
DeBiasio and Molla~\cite{DM}. They proved that, for sufficiently
large even $n$, the condition $\delta^0(G)\geq \frac{n}{2}$ guarantees an
anti-directed Hamilton cycle apart from two explicit extremal digraphs; in
particular, $\delta^0(G)\geq \frac{n}{2}+1$ always suffices.

DeBiasio and Treglown gave an asymptotic generalization of Corollary \ref{thm:DKMOT} which says that if $G$ is a sufficiently large $n$-vertex digraph with $\delta(G)\geq (1+o(1))n$, then $G$ contains every orientation of a Hamilton cycle (except perhaps the directed orientation).
The only use of the regularity lemma in their proof is a strengthening of Taylor's theorem regarding  universal $k$-linkage~\cite[Theorem~3.5]{DT} (stated below).  Theorems~\ref{thm:robust-inherit}
and~\ref{thm:main} can be combined to give a regularity-free proof of this linkage result with explicit quantitative bounds.

\begin{theorem}[Universally $k$-linked]
\label{thm:klinked}
There exists an absolute constant $K_3>0$ such that the following
holds. Let $0<\beta\leq\frac{\nu^3}{64}$,
$0<\nu\leq\tau\leq\frac{\eta}{2048}<\frac{1}{2048}$, and
$n\geq\exp(\frac{K_3}{\nu^{40}})$.
Let $k\in\mathbb N$ satisfy $k\leq\frac{1}{\beta^2}$, and let $G$ be an
$n$-vertex robust $(\nu,\tau)$-outexpander satisfying
$\delta^0(G)\geq\eta n$. Let $Q_1,\ldots,Q_k$ be oriented paths with
$|Q_i|=:\ell_i\geq\frac{1}{\beta}$ for all $i\in[k]$ and
$\ell_1+\cdots+\ell_k=n$. For all choices of pairwise distinct
vertices $u_1,\ldots,u_k,v_1,\ldots,v_k\in V(G)$, there exists a family
$P_1,\ldots,P_k$ of pairwise vertex-disjoint oriented paths in $G$ such that, for
all $i\in[k]$, the path $P_i$ is a copy of $Q_i$ whose first vertex is
$u_i$ and whose last vertex is $v_i$.
\end{theorem}

By replacing the use of \cite[Theorem~3.5]{DT} with Theorem~\ref{thm:klinked} and retaining the remainder of their
proof\footnote{More concretely, the proof of
\cite[Proposition~2.1]{DT} allows one to choose the robust expansion parameter $\nu$ so that 
$\frac{1}{\nu}\leq\exp(\exp(O(\frac{1}{\eta})))$ which ulimately gives $n_0\leq\exp(\exp(\exp(O(\frac{1}{\eta}))))$.
The remaining parameter choices outside the linkage theorem have
polynomial dependence on the auxiliary parameters, so
Theorem~\ref{thm:klinked} gives $n_0\leq\exp(\exp(\exp(O(\frac{1}{\eta}))))$.
We do not attempt to optimize this bound, but there is a chance that this bound on $n_0$ could be improved by improving the bounds in \cite[Proposition~2.1]{DT}.} unchanged, we obtain a regularity-free proof of the following result.

\begin{corollary}[\cite{DT}]\label{thm:DT}
For all $\eta>0$, there exists a constant $n_0$ such that if $G$ is an $n$-vertex digraph with $n\geq n_0$ and $\delta(G)\geq (1+\eta)n$, then $G$ contains every orientation of a Hamilton cycle (except perhaps the directed Hamilton cycle).
\end{corollary}

The same substitution in the proof of~\cite[Theorem~1.6]{DT}, with the
rest of their argument unchanged, gives a regularity-free proof of the
following pancyclicity generalization of Corollary~\ref{thm:DT}.

\begin{corollary}[\cite{DT}]\label{cor:DT-pancyclicity}
For all $\eta>0$, there exists a constant $n_0$ such that if $G$ is an
$n$-vertex digraph with $n\geq n_0$ and $\delta(G)\geq(1+\eta)n$, then,
for all $3\leq \ell\leq n$, the digraph $G$ contains every orientation
of a cycle of length $\ell$, except perhaps the directed cycle.
\end{corollary}

\subsection{Overview}
\label{sec:overview}

Section~\ref{sec:motivation} explains how the inheritance problem arises
in the proof of H\"aggkvist and Thomason and how the container method
extends their argument. In Section~\ref{sec:robust-lemmas}, we collect
some facts about robust expansion and construct the short paths and
sorters used later.

In Section~\ref{sec:hamilton}, we use containers for empty ordered pairs
to prove the Hall-type inheritance theorem
(Theorem~\ref{thm:hall-inherit}). We then construct suitable pipelines
and adapt the H\"aggkvist--Thomason construction to prove universal
Hamilton-connectivity (Theorem~\ref{thm:main}).
Section~\ref{sec:linkage} extends the container argument to sparse pairs.
A weighted union bound gives the robust-inheritance theorem
(Theorem~\ref{thm:robust-inherit}), which allows us to adapt the
DeBiasio--Treglown argument and prove universal linkage
(Theorem~\ref{thm:klinked}).

Section~\ref{sec:concluding} discusses further applications and an
inheritance problem for sparse expanders.
Appendix~\ref{app:alternate-inheritance} gives two alternate proofs of
the inheritance conclusions, with different quantitative bounds. One
uses Frieze--Kannan weak regularity and the cut-norm sampling theorem
of Alon, Fernandez de la Vega, Kannan, and Karpinski~\cite{AFKK}.
The other uses the partition-testing theorem of Fischer, Matsliah,
and Shapira~\cite{FMS}.

\subsection{Notation}

All digraphs in this paper are finite and have no loops. They may contain
both directed edges $xy$ and $yx$; an \emph{oriented graph} is a digraph
in which, for all distinct vertices $x,y$, at most one of $xy$ and $yx$
is present. We write $V(G)$ and $E(G)$ for the vertex and edge sets of a
digraph $G$, set $|G|=|V(G)|$, and write $G[U]$ for the subdigraph induced
by $U\subseteq V(G)$. An assertion involving the symbol $\pm$ is understood to hold separately
with each choice of sign. A path of length $k$ has $k$ edges and $k+1$
vertices.

For $v\in V(G)$, let
$N_G^+(v)=\{w\in V(G):vw\in E(G)\}$ and $N_G^-(v)=\{w\in V(G):wv\in E(G)\}$.
Set $d_G^\pm(v)=|N_G^\pm(v)|$, $\delta^\pm(G)=\min_{v\in V(G)}d_G^\pm(v)$, and $\delta^0(G)=\min\{\delta^+(G),\delta^-(G)\}$.
The total degree of $v$ is
$d_G(v)=d_G^+(v)+d_G^-(v)$, and
$\delta(G)=\min_{v\in V(G)}d_G(v)$.
For $A\subseteq V(G)$, define
$N_G^\pm(A)=\bigcup_{v\in A}N_G^\pm(v)$. For sets $A,B\subseteq V(G)$, let $e_G(A,B)$ denote the number of directed
edges from $A$ to $B$.
We omit the subscript $G$ when the ambient
digraph is clear.

If $V$ is a finite set, then $\binom{V}{s}$ denotes
the family of its $s$-element subsets. All logarithms are natural unless a
base is displayed explicitly. A sequence of events holds \emph{with high
probability}, or \emph{whp}, if its probability tends to one as
$n\to\infty$. We omit floors and ceilings when the resulting rounding
error is harmless, and retain them whenever they affect a one-sided
inequality or an explicit bound.

\section{Motivation from H\"aggkvist and Thomason}
\label{sec:motivation}

The following passages explain why Theorem~\ref{thm:hall-inherit} is the
main new ingredient needed for Theorem~\ref{thm:main}. H\"aggkvist and
Thomason observe that once a suitable pipeline has been found, the rest of
their proof is already in place. We give the resulting adaptation in
Section~\ref{sec:HT}. The emphasis is mine.

From the first two paragraphs of Section~6 on page~347 of~\cite{HT}:
\begin{quote}
Our aim in this section is to show how to find a pipeline within an
oriented graph. In fact, we shall show that a randomly chosen sequence of
sets $S_i$ form a pipeline. This cannot happen if
$\delta^0(G)<(\frac{3}{8}-\varepsilon)|G|$, since examples of such graphs exist
which are not expanders, and in that case a randomly chosen pair of subsets
is very unlikely to be $\lambda$-expanding. \emph{We believe that if
$\delta^0(G)>(\frac{3}{8}+\varepsilon)|G|$, in which case the graph is an expander
(by Lemma 1), it is likely that a randomly chosen sequence of sets will
yield a pipeline. Given such a pipeline, the machinery for proving our main
theorem for graphs with this value of $\delta^0$ is all in place.}

\emph{However, our present proof that randomly chosen sets form matched pairs
(and so, by Theorem 6, a pipeline) does not make full use of the expansion
properties of $G$.} Rather, we achieve the required effect by using only
properties of the neighbourhoods of pairs of vertices. In consequence, our
proof works only for $\delta^0(G)>(\frac{5}{12}+\varepsilon)|G|$. We hope to have
the opportunity to repair this deficiency in the future.
\end{quote}

From the first paragraph of page~340 of~\cite{HT}:
\begin{quote}
The proof of our theorem is based on the expansion properties of a graph
with large minimum degree. As such, \emph{we expect that the machinery developed... could be used to prove similar results for
any directed graph having a suitable expansion property...} It is our intention to explore
elsewhere a possible extension to a more general context.
\end{quote}

The second passage also anticipates the extension to digraphs with a general
expansion hypothesis in Theorem~\ref{thm:main}.

Our proof of Theorem~\ref{thm:hall-inherit} can be viewed as an extension of
the certificate argument in~\cite[Lemma~8]{HT}. Under their stronger
semidegree hypothesis, H\"aggkvist and Thomason show that a failure of
expansion can be certified by at most two vertices. They bound the
probability that a random pair of sets fails to expand by taking a union
bound over these certificates.

For a general robust expander, there need not be such a two-vertex
certificate. We instead use the Kleitman--Winston algorithm~\cite{KW},
selecting the first vertex in a fixed
order whose current degree is at least a prescribed threshold.
It assigns every empty pair arising from a
failure of expansion to a small fingerprint and one of polynomially many
container pairs. These fingerprints play the same role as the certificates in the
H\"aggkvist--Thomason proof: they reduce the possible obstructions to a
family over which a union bound is effective. Graph-container methods
were systematically developed by Sapozhenko~\cite{Sapozhenko}, who in particular coined the term ``container.'' The general hypergraph-container
method was developed independently by Balogh, Morris, and
Samotij~\cite{BMS} and by Saxton and
Thomason~\cite{ST}.

\section{Lemmas related to robust expanstion}
\label{sec:robust-lemmas}

The following elementary consequence of robust expansion is implicit in
the proof of~\cite[Lemma~5.2]{KO}. We include a proof because the statement
in~\cite{KO} is not given in exactly this form.

\begin{lemma}\label{lem:density}
Let $0<\nu\leq\tau\leq\frac{\eta}{2}<\frac{1}{2}$, and let $G$ be an $n$-vertex robust
$(\nu,\tau)$-outexpander satisfying
$\dmin(G)\geq\eta n$.
If $P,Q\subseteq V(G)$ satisfy $|P|,|Q|\geq\tau n$ and
$|P|+|Q|\geq(1-\frac{\nu}{2})n$, then
$e_G(P,Q)\geq\frac{\nu^2n^2}{4}$.
\end{lemma}

\begin{proof}
Suppose first that $\tau n\leq |P|\leq(1-\tau)n$.
Every vertex in $\RN_\nu(P)$ receives at least $\nu n$ edges from $P$.
Consequently,
$|Q\cap\RN_\nu(P)|\leq \frac{e_G(P,Q)}{\nu n}$.
If $e_G(P,Q)<\frac{\nu^2n^2}{4}$, robust outexpansion gives
$$
 |Q|
\leq n-|\RN_\nu(P)|+|Q\cap\RN_\nu(P)|
<n-|P|-\nu n+\frac{\nu}{4}n
\leq n-|P|-\frac{3\nu}{4}n,
$$
contrary to $|P|+|Q|\geq(1-\frac{\nu}{2})n$.

It remains to consider $|P|>(1-\tau)n$. Every vertex of $Q$ has at least
$\eta n$ inneighbors in $G$, while fewer than $\tau n$ vertices lie
outside $P$. Hence
$$
 e_G(P,Q)
 \geq |Q|(\eta-\tau)n
 \geq\tau(\eta-\tau)n^2
 \geq\nu^2n^2.
$$
The last inequality follows from
$\eta-\tau\geq\tau\geq\nu$.
\end{proof}

For $X\subseteq V(G)$ and $\mu>0$, define the \emph{$\mu$-robust
inneighborhood} of $X$ by
$RN^-_\mu(X)=\{v\in V(G):|N^+(v)\cap X|\geq\mu n\}$.
We say that $G$ is a \emph{robust $(\mu,\sigma)$-inexpander} if
$|RN^-_\mu(X)|\geq|X|+\mu n$ whenever
$\sigma n\leq|X|\leq(1-\sigma)n$. A digraph which is both a robust
$(\mu,\sigma)$-outexpander and a robust $(\mu,\sigma)$-inexpander is a
\emph{robust $(\mu,\sigma)$-diexpander}.

The following result, due to DeBiasio and communicated to Taylor, appears
without explicit quantification in~\cite[Proposition~48]{Taylor}.

\begin{lemma}\label{lem:inexpansion}
Let $0<\nu\leq\tau\leq\frac{\eta}{2}<\frac{1}{2}$. If $G$ is an $n$-vertex robust
$(\nu,\tau)$-outexpander satisfying
$\dmin(G)\geq\eta n$, then $G$ is a robust
$(\nu^3,2\tau)$-inexpander.
\end{lemma}

\begin{proof}
Fix $X\subseteq V(G)$ with
$2\tau n\leq|X|\leq(1-2\tau)n$, and set
$R=RN^-_{\nu^3}(X)$. Since
$e_G(V(G),X)\geq\eta n|X|$, while
$e_G(V(G),X)\leq |R||X|+\nu^3n^2$, we have $|R|>\frac{\eta n}{2}$.
Indeed, otherwise the upper bound would be at most
$\frac{\eta n|X|}{2}+\nu^3n^2<\eta n|X|$; here
$\nu^3\leq\nu\tau\leq\frac{\eta\tau}{2}\leq\frac{\eta|X|}{4n}$,
where we used $|X|\geq2\tau n$.

If $|R|<|X|+\nu^3n$ and $T=V(G)\setminus R$, then
$$
 |T|>n-|X|-\nu^3n>\tau n
 ~\text{and}~
 |T|<n-\frac{\eta n}{2}\leq(1-\tau)n,
$$
so robust outexpansion applies to $T$. On the
other hand, $e_G(T,X)<|T|\nu^3n$, and hence
$|X\cap\RN_\nu(T)|<|T|\nu^2\leq\nu^2n$. Consequently,
$$
 |\RN_\nu(T)|
 <n-|X|+\nu^2n
 <n-|X|-\nu^3n+\nu n
 <|T|+\nu n,
$$
contrary to robust outexpansion. Here the middle inequality follows from
$\nu+\nu^2<1$.
\end{proof}

The expansion assertion in the following lemma is
\cite[Proposition~35]{Taylor}.

\begin{lemma}\label{lem:bounded-deletion}
Let $0<\nu\leq\tau<\frac{1}{4}$, and let $G$ be an $n$-vertex robust
$(\nu,\tau)$-outexpander. If $X\subseteq V(G)$ satisfies
$|X|\leq\frac{\nu n}{4}$, then $G-X$ is a robust
$(\frac{\nu}{2},2\tau)$-outexpander. The same statement holds for robust
inexpanders and robust diexpanders. Moreover, if
$\eta\geq\nu$ and $\delta^0(G)\geq\eta n$, then
$\delta^0(G-X)\geq\frac{3\eta}{4}|G-X|$.
\end{lemma}

\begin{proof}
Let $H=G-X$ and $N=|H|$. Since $N\geq\frac{n}{2}$, every set
$S\subseteq V(H)$ with $2\tau N\leq|S|\leq(1-2\tau)N$ satisfies
$\tau n\leq|S|\leq(1-\tau)n$. Hence
$$
 |\RN_{\nu/2,H}(S)|
 \geq|\RN_{\nu,G}(S)|-|X|
 \geq|S|+\nu n-|X|
 \geq|S|+\frac{\nu N}{2}.
$$
Reversing all edge directions proves the assertion for inexpanders,
and the diexpander assertion follows by applying both conclusions.
If $\delta^0(G)\geq\eta n$ and $\eta\geq\nu$, then
$\delta^0(H)\geq\eta n-|X|\geq\frac{3\eta n}{4}
\geq\frac{3\eta N}{4}$.
\end{proof}

The next lemma is an explicitly quantified version of
\cite[Lemma~58]{Taylor} (see also \cite[Lemma~5.1]{DT}) which gives short paths of every orientation between every pair of vertices. We also allow a
set of forbidden vertices and several copies with pairwise vertex-disjoint
interiors.

\begin{lemma}\label{lem:fixed-endpoint-paths}
Let $0<\nu\leq\tau\leq\frac{\eta}{8}<\frac{1}{8}$ and
$\frac{2}{\nu^3}\leq k\leq\frac{\nu^3n}{4}$.
If $G$ is an $n$-vertex robust $(\nu,\tau)$-outexpander satisfying
$\dmin(G)\geq\eta n$, then, for all distinct vertices
$x,y\in V(G)$ and all oriented paths $P$ of length $k$, there exists a copy
of $P$ in $G$ whose first and last vertices are $x$ and $y$, respectively.
More generally, if $D\subseteq V(G)\setminus\{x,y\}$ and
$q\in\mathbb N$ satisfy $|D|+qk\leq\frac{\nu^3n}{4}$,
then $G-D$ contains $q$ copies of $P$ from $x$ to $y$ whose interiors
are pairwise vertex-disjoint.
\end{lemma}

\begin{proof}
By Lemma~\ref{lem:inexpansion}, $G$ is a robust
$(\nu^3,2\tau)$-diexpander. Set $\beta=\frac{\nu^3}{2}$ and
$r=\lceil\frac{1}{\beta}\rceil\leq k$, and write the vertices of $P$,
in their underlying order, as $p_0,p_1,\ldots,p_k$.

Let $Z\subseteq V(G)\setminus\{x,y\}$ satisfy
$|Z|+k\leq\frac{\nu^3n}{4}$. We construct a copy of $P$ from $x$ to $y$
avoiding $Z$. Starting at $x$, embed $p_0p_1\cdots p_{k-r}$ greedily in
$G-Z$, always reserving $y$. This is possible since at each step the
required in- or outneighborhood has order at least $\eta n$, whereas
at most $|Z|+k\leq\frac{\nu^3n}{4}<\eta n$ vertices are unavailable.
Let $z$ be the image
of $p_{k-r}$. Delete $Z$ and all vertices of this initial segment except
$z$, leaving a digraph $H$ of order $N>\frac{n}{2}$. At most
$\frac{\nu^3n}{4}$ vertices have been deleted, so
Lemma~\ref{lem:bounded-deletion}, applied with expansion parameters
$\nu^3$ and $2\tau$, shows that $H$ is a robust
$(\beta,4\tau)$-diexpander with $\delta^0(H)\geq\gamma N$, where
$\gamma=\frac{3\eta}{4}$. We have $\gamma>4\tau+\beta$ and
$\beta N>k\geq r$.

We embed the remaining $r$ edges as in \cite[Lemma~58]{Taylor}.
Let $A_1$ be the in- or outneighborhood of
$z$ in $H$ prescribed by the edge $p_{k-r}p_{k-r+1}$. For
$1\leq i\leq r-2$, let $A_{i+1}=RN^+_{\beta,H}(A_i)$ if the edge
$p_{k-r+i}p_{k-r+i+1}$ is forward, and let
$A_{i+1}=RN^-_{\beta,H}(A_i)$ otherwise. Every vertex of $A_{i+1}$ has
at least $\beta N$ choices for the preceding vertex in $A_i$.

We have $|A_1|\geq\gamma N>4\tau N$. Whenever
$4\tau N\leq|A_i|\leq(1-4\tau)N$, robust diexpansion gives
$|A_{i+1}|\geq|A_i|+\beta N$. If $|A_i|>(1-4\tau)N$, every vertex of
$H$ has more than $(\gamma-4\tau)N>\beta N$ appropriate neighbors in
$A_i$, so $A_{i+1}=V(H)$. Consequently,
$$
 |A_{r-1}|\geq\min\{N,(\gamma+(r-2)\beta)N\}
 \geq(1-\gamma+\beta)N,
$$
where the last inequality follows from $r\beta\geq1$ and
$\gamma\geq2\beta$.

The appropriate in- or outneighborhood of $y$ for the final edge of $P$
has order at least $\gamma N$, so at least $\beta N$ vertices of
$A_{r-1}$ can be the penultimate vertex. Choose one distinct from $z$
and $y$, and choose the preceding vertices backwards through
$A_{r-2},\ldots,A_1$, always avoiding $z$, $y$, and the vertices already
chosen. At most $r$ vertices are excluded at any step, while at least
$\beta N>r$ choices are available. This gives a path from $z$ to $y$
with the required orientation, which together with the initial segment
gives a copy of $P$ avoiding $Z$.

For the more general assertion, start with $Z=D$ and successively add
the interiors of the completed copies to $Z$. Before the $j$th copy,
$|Z|+k\leq|D|+jk\leq|D|+qk\leq\frac{\nu^3n}{4}$, so the construction
applies at every step. The first assertion follows by taking
$D=\varnothing$ and $q=1$.
\end{proof}

\begin{definition}[$(s,A)$-sorter]
Let $s\in\mathbb N$ and let $A$ be an oriented path. An
\emph{$(s,A)$-sorter} is a digraph on $s\cdot |V(A)|$ vertices with designated input vertices
$a_1,\ldots,a_s$ and output vertices $a'_1,\ldots,a'_s$ such that, for
all permutations $\pi$ of $[s]$, its vertex set can be partitioned into
$s$ vertex-disjoint copies of $A$ joining $a_i$ to $a'_{\pi(i)}$, for all $i\in[s]$.
\end{definition}

We recall the sorting networks used in \cite[Section~4]{HT}. A
\emph{comparator} compares the labels in two positions and sends the
smaller label to the lower-numbered position and the larger label to the
other. A \emph{sorting network} on $s$ positions is a fixed sequence of
stages, each consisting of comparators on disjoint pairs of positions,
which sorts all permutations of $[s]$. Positions not used by a comparator
retain their labels. Batcher~\cite{Batcher} constructed such networks
with at most $\lceil(\log_2s)^2\rceil$ stages. If we allow each comparator
either to exchange its two labels or to leave them in place, the network
can realize any permutation $\pi$: assign label $\pi(i)$ to input $i$
and sort, so that this label reaches output $\pi(i)$.

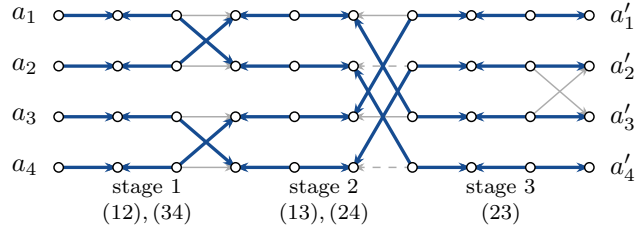
\begin{figure}[htbp]
\centering
\begin{tikzpicture}[
  x=0.78cm,
  y=0.67cm,
  edge/.style={gray!65,line width=0.55pt,-{Stealth[length=1.45mm]}},
  edashed/.style={gray!65,dashed,line width=0.55pt,-{Stealth[length=1.45mm]}},
  route/.style={reviewblue,line width=1.25pt,-{Stealth[length=1.7mm]}},
  vertex/.style={circle,draw=black,fill=white,line width=0.55pt,
                 inner sep=1.25pt}
]
  \foreach \j in {0,...,9} {
    \foreach \i in {1,2,3,4} {
      \pgfmathsetmacro{\yy}{4-\i}
      \coordinate (v\j-\i) at (\j,\yy);
    }
  }

  \foreach \i in {1,2,3,4} {
    \draw[edge] (v0-\i)--(v1-\i);
    \draw[edge] (v2-\i)--(v1-\i);
  }
  \foreach \a in {1,2} {
    \foreach \b in {1,2} {\draw[edge] (v2-\a)--(v3-\b);}
  }
  \foreach \a in {3,4} {
    \foreach \b in {3,4} {\draw[edge] (v2-\a)--(v3-\b);}
  }

  \foreach \i in {1,2,3,4} {
    \draw[edge] (v4-\i)--(v3-\i);
    \draw[edge] (v4-\i)--(v5-\i);
  }
  \foreach \a in {1,3} {
    \foreach \b in {1,3} {\draw[edge] (v6-\a)--(v5-\b);}
  }
  \foreach \a in {2,4} {
    \foreach \b in {2,4} {\draw[edashed] (v6-\a)--(v5-\b);}
  }

  \foreach \i in {1,2,3,4} {
    \draw[edge] (v6-\i)--(v7-\i);
    \draw[edge] (v8-\i)--(v7-\i);
  }
  \draw[edge] (v8-1)--(v9-1);
  \draw[edge] (v8-4)--(v9-4);
  \foreach \a in {2,3} {
    \foreach \b in {2,3} {\draw[edge] (v8-\a)--(v9-\b);}
  }

  \foreach \i in {1,2,3,4} {
    \draw[route] (v0-\i)--(v1-\i);
    \draw[route] (v2-\i)--(v1-\i);
  }
  \draw[route] (v2-1)--(v3-2);
  \draw[route] (v2-2)--(v3-1);
  \draw[route] (v2-3)--(v3-4);
  \draw[route] (v2-4)--(v3-3);
  \foreach \i in {1,2,3,4} {
    \draw[route] (v4-\i)--(v3-\i);
    \draw[route] (v4-\i)--(v5-\i);
  }
  \draw[route] (v6-1)--(v5-3);
  \draw[route] (v6-3)--(v5-1);
  \draw[route] (v6-2)--(v5-4);
  \draw[route] (v6-4)--(v5-2);
  \foreach \i in {1,2,3,4} {
    \draw[route] (v6-\i)--(v7-\i);
    \draw[route] (v8-\i)--(v7-\i);
    \draw[route] (v8-\i)--(v9-\i);
  }
  \foreach \j in {0,...,9} {
    \foreach \i in {1,2,3,4} {\node[vertex] at (v\j-\i) {};}
  }
  \foreach \i in {1,2,3,4} {
    \pgfmathsetmacro{\yy}{4-\i}
    \node[anchor=east,font=\small] at (-0.18,\yy) {$a_{\i}$};
    \node[anchor=west,font=\small] at (9.18,\yy) {$a'_{\i}$};
  }

  \node[font=\scriptsize,align=center] at (1.5,-0.68)
    {stage 1\\$(12),(34)$};
  \node[font=\scriptsize,align=center] at (4.5,-0.68)
    {stage 2\\$(13),(24)$};
  \node[font=\scriptsize,align=center] at (7.5,-0.68)
    {stage 3\\$(23)$};
\end{tikzpicture}
\caption{A $(4,A)$-sorter where $A$ is an anti-directed path of length 9. All displayed arcs belong to the sorter. The solid and dashed gray
arcs in stage~2 distinguish its two switches, while the blue arcs highlight
a decomposition into four copies of $A$ which routes $(a_1,a_2,a_3,a_4)$ to $(a_4', a_3', a_2', a_1')$.}
\label{fig:four-wire-sorter}
\end{figure}

\begin{samepage}
The next lemma is the generalization of
\cite[Theorem~4]{HT} for constructing the required sorters in robust-expanders. We use their construction with Lemma~\ref{lem:fixed-endpoint-paths} in place of \cite[Lemma~2]{HT}.

\begin{lemma}\label{lem:sorters}
Let $0<\nu\leq\tau\leq\frac{\eta}{8}<\frac{1}{8}$, let
$n\geq\exp(\frac{1}{\nu^4})$, and let $G$ be an $n$-vertex robust
$(\nu,\tau)$-outexpander with $\delta^0(G)\geq\eta n$. Set
$$
 s=\left\lceil(\log n)^9\right\rceil,~
 q=\left\lceil\frac{1}{\nu^3}\right\rceil,~
 d=\left\lceil(\log_2s)^2\right\rceil,~
 \ell=(q+1)d.
$$
For all oriented paths $A,D$ of length $\ell$ and all sets
$Z\subseteq V(G)$ with $|Z|\leq 2$, $G-Z$ contains vertex-disjoint
$(s,A)$- and $(s,D)$-sorters. Each sorter has $s(\ell+1)$ vertices.
\end{lemma}
\end{samepage}

\begin{proof}
Set $h=\lceil\frac{16}{\eta^2}\rceil$ and $b=\lfloor\frac{\eta n}{2}\rfloor$.
Since $\log n\geq\frac{1}{\nu^4}>4096$, we have
$s\leq2(\log n)^9$, $q+1,h\leq\log n$, and $d\leq(\log n)^2$. Hence
$$
 4+2s(\ell+1)+2hq
 \leq9(\log n)^{12}
 \leq\frac{n}{4\log n}
 \leq\frac{\nu^3n}{4}.
$$

Take two copies of Batcher's network on $s$ positions, adding stages
without comparators to obtain exactly $d$ stages. Split $A$ and $D$ into
$d$ consecutive subpaths of length $q+1$. In a stage corresponding to a
subpath $B$, we will replace each comparator by a $(2,B)$-sorter and
each position not used by a comparator by a copy of $B$, as illustrated
in Figure~\ref{fig:four-wire-sorter}. Choose $2s$ distinct input vertices
outside $Z$ and construct the two sorters stage by stage, keeping their
vertex sets disjoint.

Consider a comparator whose two current vertices are $x$ and $y$, and
let $Y$ consist of $Z$ and all vertices already retained in either
sorter, except $x$ and $y$. Form an oriented path of length $2q$ by
identifying the last vertices of two copies of the first $q$ edges of
$B$. Since $2q\geq\frac{2}{\nu^3}$ and
$|Y|+2hq\leq4+2s(\ell+1)+2hq\leq\frac{\nu^3n}{4}$,
Lemma~\ref{lem:fixed-endpoint-paths}, applied with path length $2q$,
forbidden set $Y$, and $h$ requested copies, gives $h$ such paths from
$x$ to $y$ whose interiors are pairwise disjoint. Let
$z_1,\ldots,z_h$ be their middle vertices.

If the last edge of $B$ is forward, consider the outneighbors of the
$z_i$; if it is backward, consider their inneighbors. After excluding
$Y$ and all vertices of the $h$ auxiliary paths, including $x$ and $y$,
each $z_i$ still has at least
$$
 \eta n-|Y|-2hq-2
 \geq\eta n-\frac{\nu^3n}{4}-2
 \geq\frac{\eta n}{2}\geq b
$$
such neighbors. Since $\eta n\geq8$, we have
$\binom b2\geq\frac{\eta^2n^2}{32}$ and hence
$h\binom b2\geq\frac{n^2}{2}>\binom n2$. Thus two distinct middle
vertices, say $z_i,z_j$, have two common appropriate neighbors $x',y'$.
Retain the segment from $x$ to $z_i$ on one auxiliary path and the
segment from $y$ to $z_j$ on the other. These are vertex-disjoint copies
of the first $q$ edges of $B$. Together with $x',y'$ and the four
appropriate edges, they form a $(2,B)$-sorter with inputs $x,y$ and
outputs $x',y'$. The remaining vertices of the auxiliary paths are not
retained in the construction.

At a position not used by a comparator, extend its current vertex
greedily by a copy of $B$ whose new vertices avoid $Z$ and both partially
constructed sorters.
This is possible because fewer than $\eta n$ vertices are unavailable
at any step. Each stage of either sorter adds $s(q+1)$ vertices and
permits all exchanges in the corresponding sorting-network stage. The two completed sorters
therefore have the required routing property and each has
$s+sd(q+1)=s(\ell+1)$ vertices.
\end{proof}

\section{Oriented Hamilton cycles in robust expanders}
\label{sec:hamilton}

We first prove Theorem~\ref{thm:hall-inherit} and then adapt the
H\"aggkvist--Thomason argument to prove Theorem~\ref{thm:main}.

\subsection{Containing empty pairs}

For a digraph $G$, its \emph{bipartite representation} $H_G$ has vertex
classes $V_L=\{v_L:v\in V(G)\}$ and
$V_R=\{v_R:v\in V(G)\}$,
and $u_Lv_R\in E(H_G)$ precisely when $uv\in E(G)$. For
$X,Y\subseteq V(G)$, set $X_L=\{x_L:x\in X\}$ and
$Y_R=\{y_R:y\in Y\}$.
Thus
\begin{equation}\label{eq:independent-equivalence}
 e_G(X,Y)=0
 \Longleftrightarrow
 X_L\cup Y_R\text{ is independent in }H_G.
\end{equation}

We can therefore apply graph-container lemmas to $H_G$.

\begin{definition}[Container families]\label{def:container-system}
Let $G$ be a digraph on $n$ vertices and let $\alpha,\beta>0$.
A family $\cC\subseteq 2^{V(G)}\times 2^{V(G)}$
is an \emph{$(\alpha,\beta)$-container family for $G$} if the
following properties hold:
\begin{enumerate}
\item for all $X,Y\subseteq V(G)$ with $e_G(X,Y)=0$, there exists
      $(P,Q)\in\cC$ such that $X\subseteq P$ and $Y\subseteq Q$;
\item for all $(P,Q)\in\cC$, at least one of the following inequalities holds:
\begin{equation}\label{eq:container-alternatives}
 |P|<\alpha n,~ |Q|<\alpha n\text{,~or~}
 |P|+|Q|<(1-\beta)n.
\end{equation}
\end{enumerate}
\end{definition}

We prove the following lemma using the Kleitman--Winston
algorithm~\cite{KW}.

\begin{lemma}[Independent-set container lemma]\label{lem:threshold-container}
Let $H$ be a graph on $N$ vertices and let $d>0$.
There exists a family
$\mathcal F\subseteq 2^{V(H)}$ together with a map
$C\colon\mathcal F\to 2^{V(H)}$ such that
\begin{enumerate}
\item for all independent sets $I\subseteq V(H)$, there exists
      $F\in\mathcal F$ such that $F\subseteq I\subseteq C(F)$;
\item for all $F\in\mathcal F$, we have
      $|F|\leq \frac{N}{d}$ and $\Delta(H[C(F)])<d$.
\end{enumerate}
Consequently,
$|\mathcal F|\leq\sum_{j=0}^{\lfloor N/d\rfloor}\binom Nj$.
\end{lemma}

\begin{proof}
Fix a linear ordering $\prec$ of $V(H)$. For all independent sets
$I\subseteq V(H)$, we define sets $U_i$ and $F_i$ recursively. Begin with
$U_0=V(H)$ and $F_0=\emptyset$.
Suppose that $U_i$ and $F_i$ have been defined. If
$\Delta(H[U_i])<d$, stop and set $U_I=U_i$ and $F_I=F_i$.
Otherwise, let $v_i$ be the $\prec$-first vertex of $U_i$ whose degree in
$H[U_i]$ is at least $d$, and define
$$
 (U_{i+1},F_{i+1})=
 \begin{cases}
 \bigl(U_i\setminus(\{v_i\}\cup N(v_i)),
       F_i\cup\{v_i\}\bigr),&v_i\in I,\\[2mm]
 \bigl(U_i\setminus\{v_i\},F_i\bigr),&v_i\notin I.
 \end{cases}
$$
At least one vertex is removed at every step, so the algorithm terminates. Note that $F_I\subseteq I$.

Whenever a vertex is added to $F_I$, at least $d$ vertices of the current
set $U_i$ are deleted as its neighbors. The sets deleted at different such
steps are disjoint, and hence $|F_I|\leq \frac{N}{d}$.

We claim that $U_I$ is determined by $F_I$. Indeed, rerun the algorithm
with $F_I$ in place of $I$, and write $U_i'$ for the current set in this
second run. We show inductively that $U_i'=U_i$. If the two current sets
agree, then both runs select the same vertex $v_i$. If $v_i\in F_I$, then
$v_i\in I$, so both runs delete $v_i$ together with its current neighbors.
If $v_i\notin F_I$, then in fact $v_i\notin I$: otherwise the original run
would add $v_i$ to the fingerprint at this step, giving $v_i\in F_I$.
Thus both runs delete only $v_i$. In particular, a vertex of
$I\setminus F_I$ is never selected in either run. It follows by induction
that the two runs have the same current set at every step and terminate
with the same set. Write the terminal set of the second run as $U_{F_I}$;
then $U_I=U_{F_I}$.

Let $\mathcal F$ be the family of fingerprints $F_I$ arising from
independent sets $I\subseteq V(H)$, and define $C(F)=F\cup U_F$.
If $I$ is independent, then no vertex of $I$ can be deleted as a neighbor
of a vertex added to $F_I$. Thus every vertex of $I$ is either placed in
$F_I$ or remains in $U_I$, and consequently
$F_I\subseteq I\subseteq F_I\cup U_I=C(F_I)$.

Finally, every vertex of $F$ is isolated in $H[C(F)]$: the set $F$ is
independent, and all remaining neighbors of a fingerprint vertex are
deleted when that vertex is selected. The stopping rule also gives
$\Delta(H[U_F])<d$.
It follows that $\Delta(H[C(F)])<d$. Since every fingerprint has order at
most $\frac{N}{d}$, counting its possible choices gives
$|\mathcal F|\leq\sum_{j=0}^{\lfloor N/d\rfloor}\binom Nj$.
\end{proof}

We now apply Lemma~\ref{lem:threshold-container} to the bipartite
representation of a robust expander. If a resulting container violated
all three alternatives in Definition~\ref{def:container-system},
Lemma~\ref{lem:density} would force too many edges.

\begin{corollary}[Containing empty pairs]\label{lem:containers}
Let $0<\nu\leq\tau\leq\frac{\eta}{2}<\frac{1}{2}$, and let $G$ be an $n$-vertex robust
$(\nu,\tau)$-outexpander satisfying $\dmin(G)\geq\eta n$.
If $n\geq\frac{16}{\nu^2}$, then
there exists a $(\tau,\frac{\nu}{2})$-container family $\cC$ for $G$ such that
$|\cC|\leq n^{48/\nu^2}$.
\end{corollary}

\begin{proof}
Let $H=H_G$, so that $H$ has $2n$ vertices, and set
$d=\frac{\nu^2n}{8}$.
Apply Lemma~\ref{lem:threshold-container} to $H$ with this parameter.
For all resulting fingerprints $F\in\mathcal F$, write
$C_F=C(F)$ and define
$P(F)=\{v\in V(G):v_L\in C_F\}$ and
$Q(F)=\{v\in V(G):v_R\in C_F\}$.
Thus $C_F=P(F)_L\cup Q(F)_R$. Set
$\cC=\{(P(F),Q(F)):F\in\mathcal F\}$.

If $X,Y\subseteq V(G)$ satisfy $e_G(X,Y)=0$, then $I=X_L\cup Y_R$ is
independent by~\eqref{eq:independent-equivalence}.
Lemma~\ref{lem:threshold-container} therefore gives a fingerprint
$F\in\mathcal F$ such that $I\subseteq C_F$, and the
corresponding pair satisfies
$X\subseteq P(F)$ and $Y\subseteq Q(F)$.

Fix a fingerprint $F\in\mathcal F$ and write
$P=P(F)$ and $Q=Q(F)$.
Since $\Delta(H[C_F])<d$, we have
$$
 e_G(P,Q)=e(H[C_F])\leq|P|d\leq nd<\frac{\nu^2}{4}n^2.
$$
If none of the three alternatives in~\eqref{eq:container-alternatives} held,
Lemma~\ref{lem:density} would instead give
$e_G(P,Q)\geq\frac{\nu^2n^2}{4}$, a contradiction.

It remains to bound $|\cC|$. Since $\frac{2n}{d}=\frac{16}{\nu^2}$,
Lemma~\ref{lem:threshold-container} therefore gives
$$
 |\cC|
 \leq|\mathcal F|
 \leq\sum_{j=0}^{\lfloor16/\nu^2\rfloor}\binom{2n}{j}
 \leq n^{48/\nu^2}.
$$
\end{proof}

\subsection{Inheritance of Hall-type expansion}

We first show that random sets are typical with high probability. We then
prove that every typical pair is expanding, which yields
Theorem~\ref{thm:hall-inherit}.

\begin{definition}[Typical set]\label{def:typical}
Let $G$ be a digraph on $n$ vertices, let $\cC$ be a family of ordered
pairs of subsets of $V(G)$, let $\gamma>0$, and let $S\subseteq V(G)$ have
order $s$. The set $S$ is \emph{$(\cC,\gamma)$-typical in $G$} if
\begin{align}
 |S\cap N^\pm(v)|
 &\geq\left(\frac{|N^\pm(v)|}{n}-\gamma\right)s
 &&\text{for all }v\in V(G), \label{eq:degree-typical}\\
 |S\cap R|
 &\leq\left(\frac{|R|}{n}+\gamma\right)s
 &&\text{for all }R\in\{P,Q:(P,Q)\in\cC\}. \label{eq:container-typical}
\end{align}
\end{definition}

We use the following standard form of the Chernoff--Hoeffding bound for
sampling without replacement; see~\cite{Hoeffding}.

\begin{lemma}[Chernoff--Hoeffding bound]\label{lem:chernoff-hoeffding}
Let $\Omega$ be a set of order $n$, let $R\subseteq\Omega$, let
$1\leq s\leq n$, and let $S$ be a uniformly random $s$-subset of $\Omega$.
If $X=|S\cap R|$ and
$\mu=\mathbb E X=\frac{|R|s}{n}$, then, for all $t\geq0$,
$$
 \Pr(X\leq\mu-t)\leq\exp\left(-\frac{2t^2}{s}\right)
 ~\text{and}~
 \Pr(X\geq\mu+t)\leq\exp\left(-\frac{2t^2}{s}\right).
$$
\end{lemma}

We also use two sampling observations, implicit in the proofs of
\cite[Lemma~10 and Theorem~11]{HT}. In a uniformly chosen ordered family
of disjoint subsets of prescribed orders, each set is uniform among
subsets of the ambient set of its order. Also, a uniformly chosen
$j$-subset of a uniformly chosen $s$-subset is uniform among $j$-subsets
of the ambient set. Neither linearity of expectation nor the union
bound requires independence.

The following proposition is a direct generalization of
\cite[Lemma~10]{HT}. At a transition between pipeline layers, different
paths may require edges in opposite directions. Part~(ii) ensures that
after deletions we can choose typical subsets with typical complements;
these will supply the two matchings needed in
Lemma~\ref{lem:robust-pipeline}.

\begin{proposition}
\label{prop:random-typical}
Let $G$ be a digraph on $n$ vertices, let $\cC$ be a family of ordered
pairs of subsets of $V(G)$, and let $S$ be a uniformly random $s$-subset
of $V(G)$.
\begin{enumerate}
\item For all $\gamma>0$ and $1\leq s\leq n$,
$\Pr(S\text{ is not $(\cC,\gamma)$-typical in $G$})
\leq(2n+2|\cC|)e^{-2\gamma^2s}$.
\item Let $0<\gamma\leq1$ and $K\geq1$, and suppose that
$|\cC|\leq n^K$, $s\geq(\log n)^9$, and $r$ is an integer satisfying
$$
 \log n\geq
 \max\left\{\frac{4}{\gamma^2},(K+20)^{1/3}\right\}
 ~\text{and}~
 \frac{2(K+20)}{\gamma^2}\log n
 \leq r\leq(\log n)^5.
$$
With probability at least $1-\frac{1}{n}$, the following holds simultaneously
for all $H\subseteq S$ of order $h< (\log n)^3$ and all integers $u$
with $u=0$, $u=s-h$, or $r\leq u\leq s-h-r$,
at least $\left(1-\frac{1}{2n}\right)\binom{s-h}{u}$ of the $u$-subsets of
$S\setminus H$ are
$(\cC,\gamma)$-typical in $G$.
\end{enumerate}
\end{proposition}

\begin{proof}
For all sets $N^\pm(v)$ and all sets which occur as one side of
a pair in $\cC$, Lemma~\ref{lem:chernoff-hoeffding}, applied with
$t=\gamma s$, shows that the corresponding inequality in
Definition~\ref{def:typical} fails with probability at most
$e^{-2\gamma^2s}$. There are at most $2n+2|\cC|$ such inequalities, so
the union bound proves~(i).

For~(ii), the assertion with $u=0$ is immediate because the empty set is
typical. We prove the remaining cases by adapting the argument of
\cite[Lemma~10]{HT}. Call a set $J\subseteq V(G)$ \emph{bad} if it is not
$(\cC,\frac{\gamma}{2})$-typical. If $J$ is a uniformly random $j$-subset of
$V(G)$ with $j\geq1$, then part~(i) gives
\begin{equation}\label{eq:bad-j-set}
 \Pr(J\text{ is bad})
 \leq (2n+2n^K)e^{-\gamma^2j/2}
 \leq 4n^K e^{-\gamma^2j/2}=:p_j.
\end{equation}
In particular, the choice of $r$ gives
\begin{equation}\label{eq:bad-j-polynomial}
 p_j\leq4n^{-20}~\text{whenever}~j\geq r.
\end{equation}

Set $L=(\log n)^5$. We first consider integers $j$ with
$r\leq j<L$. Let $Y_j$ denote the number of bad $j$-subsets of $S$.
By symmetry and~\eqref{eq:bad-j-polynomial},
$\mathbb E Y_j\leq4n^{-20}\binom{s}{j}$.
Consequently, Markov's inequality gives
$\Pr(Y_j>n^{-3}\binom{s}{j})\leq4n^{-17}$.
A union bound over the fewer than $L$ possible values of $j$ shows that,
with failure probability at most $4Ln^{-17}$, all these inequalities hold
simultaneously. If $h<(\log n)^3$, $H_0\in\binom{S}{h}$, and
$r\leq j<L$, then $h+j<(\log n)^3+(\log n)^5\leq \frac{s}{2}$, and hence
$$
 \frac{\binom{s}{j}}{\binom{s-h}{j}}
 =\prod_{i=0}^{j-1}\left(1+\frac{h}{s-h-i}\right)
 \leq\exp\left(\frac{2hj}{s}\right)
 \leq\exp\left(\frac{2}{\log n}\right)<2.
$$
It follows that fewer than
$\frac{1}{2n}\binom{s-h}{j}$ of the $j$-subsets of $S\setminus H_0$ are bad.

We next consider $j\geq L$. For fixed integers $h$ and $j$, let
$$
 Z_{h,j}=
 \sum_{H_0\in\binom{S}{h}}
 \bigl|\{J\in\tbinom{S\setminus H_0}{j}:J\text{ is bad}\}\bigr|.
$$
There are $\binom{s}{h}\binom{s-h}{j}$ ordered pairs of disjoint
$h$- and $j$-subsets of $S$. By symmetry and~\eqref{eq:bad-j-set},
$\mathbb E Z_{h,j}\leq \binom{s}{h}\binom{s-h}{j}p_j$.
If there exists $H_0\in\binom{S}{h}$ leaving more than
$\frac{1}{2n}\binom{s-h}{j}$ bad $j$-subsets, then
$Z_{h,j}>\frac{1}{2n}\binom{s-h}{j}$. Markov's inequality therefore shows
that the probability of this event is at most $2n\binom{s}{h}p_j$.
Since $s\leq n$, a union bound over all
$h<(\log n)^3$ and $L\leq j\leq s-h$, together with
\eqref{eq:bad-j-set}, bounds the total failure probability by
$8n^{K+3}\exp((\log n)^4-\frac{\gamma^2(\log n)^5}{2})\leq8n^{-17}$.
Indeed, there are at most $n^2$ choices of $(h,j)$,
$\binom{s}{h}\leq\exp((\log n)^4)$, Markov's inequality contributes the
factor $2n$, and \eqref{eq:bad-j-set} contributes
$4n^K\exp(-\frac{\gamma^2(\log n)^5}{2})$.
The assumptions give
$\frac{\gamma^2\log n}{2}\geq2$ and $(\log n)^3\geq K+20$.
Since $4((\log n)^5+2)n^{-17}<\frac{1}{n}$, we conclude that, with probability
at least $1-\frac{1}{n}$, simultaneously for all integers
$0\leq h<(\log n)^3$, all $H_0\in\binom{S}{h}$, and all integers $j$
with $j=s-h$ or $r\leq j\leq s-h-r$,
all but at most $\frac{1}{2n}\binom{s-h}{j}$ of the $j$-subsets of
$S\setminus H_0$ are $(\cC,\frac{\gamma}{2})$-typical. Such sets are also
$(\cC,\gamma)$-typical, which proves~(ii).
\end{proof}

\begin{samepage}
The following theorem is the robust-expander generalization of
\cite[Lemma~8]{HT}.

\begin{theorem}\label{thm:lemma8}
Let $0<\nu\leq\tau\leq\frac{\eta}{8}<\frac{1}{8}$, and let $G$ be an $n$-vertex robust
$(\nu,\tau)$-outexpander satisfying
$\dmin(G)\geq\eta n$. Let $\cC$ be a
$(\tau,\frac{\nu}{2})$-container family for $G$.
If $S,S'\subseteq V(G)$ are disjoint
$(\cC,\frac{\nu}{24})$-typical sets of the same order, then $(S,S')$ is
$\frac{\nu}{4}$-expanding.
\end{theorem}
\end{samepage}

\begin{proof}
Write $|S|=|S'|=s$. Let $A\subseteq S$ with
$0<|A|=a\leq \frac{s}{2}$. We first prove
\begin{equation}\label{eq:out-expansion}
 |N^+(A)\cap S'|\geq(1+\frac{\nu}{4})a.
\end{equation}

If $a\leq\frac{\eta s}{4}$, choose $x\in A$. By
\eqref{eq:degree-typical} and
$d_G^+(x)\geq\eta n$,
$$
 |N^+(A)\cap S'|
 \geq |N^+(x)\cap S'|
 \geq(\eta-\frac{\nu}{24})s
 \geq\frac{7\eta}{8}s
 >(1+\frac{\nu}{4})a.
$$

Suppose that $a>\frac{\eta s}{4}$ and that \eqref{eq:out-expansion} fails. Set
$B=S'\setminus N^+(A)$.
Thus $e_G(A,B)=0$ and
\begin{equation}\label{eq:B-large}
 |B|>s-(1+\frac{\nu}{4})a
 \geq\frac{1-\nu/4}{2}s
 \geq\frac{3}{8}s.
\end{equation}
Choose $(P,Q)\in\cC$ with $A\subseteq P$ and $B\subseteq Q$.

If $|P|<\tau n$, then
$|A|\leq|S\cap P|\leq(\tau+\frac{\nu}{24})s
\leq\frac{3\eta}{16}s<\frac{\eta}{4}s$, a contradiction.
Likewise, $|Q|<\tau n$ would give
$|B|\leq|S'\cap Q|\leq(\tau+\frac{\nu}{24})s
\leq\frac{3\eta}{16}s<\frac{3}{8}s$, contrary to
\eqref{eq:B-large}. Therefore
$|P|+|Q|<(1-\frac{\nu}{2})n$. Applying \eqref{eq:container-typical} to $S$ and
$S'$ gives
$$
 |A|+|B|
 \leq |S\cap P|+|S'\cap Q|
 \leq\left(\frac{|P|+|Q|}{n}+\frac{\nu}{12}\right)s
 <\left(1-\frac{5\nu}{12}\right)s.
$$
On the other hand, the definition of $B$ and the assumed failure of
\eqref{eq:out-expansion} imply
$|A|+|B|>s-\frac{\nu}{4}a\geq(1-\frac{\nu}{8})s$,
a contradiction. This proves out-expansion from $S$ into $S'$.

For in-expansion, the same one-vertex estimate applies with $N^-(x)$
when $a\leq\frac{\eta s}{4}$. Otherwise, if
$|N^-(A)\cap S'|<(1+\frac{\nu}{4})a$, set $B=S'\setminus N^-(A)$.
We have $e_G(B,A)=0$, so choose a container $(P,Q)\in\cC$ with
$B\subseteq P$ and $A\subseteq Q$. Applying the preceding estimates to
$S'\cap P$ and $S\cap Q$ gives the same contradictory bounds on
$|A|+|B|$. Interchanging $S$ and $S'$ proves the remaining two inequalities,
so
$(S,S')$ is $\frac{\nu}{4}$-expanding.
\end{proof}

\begin{proof}[Proof of Theorem~\ref{thm:hall-inherit}]
Apply Corollary~\ref{lem:containers} to obtain a family $\cC$ of at most
$n^{48/\nu^2}$ ordered pairs.
Each of $S$ and $S'$ is a uniformly random $m$-subset of $V(G)$.
Proposition~\ref{prop:random-typical}(i) and a union bound show that
the probability that either set is not $(\cC,\frac{\nu}{24})$-typical is at most
$$
 2(2n+2n^{48/\nu^2})e^{-\nu^2m/288}
 \leq8n^{48/\nu^2}e^{-\nu^2m/288}.
$$
Theorem
\ref{thm:lemma8} then makes $(S,S')$ $\frac{\nu}{4}$-expanding.
\end{proof}

Let $s,t\in\mathbb N$. A \emph{pipeline of width $s$ and length $t$}
consists of pairwise disjoint $s$-sets $S_0,S_1,\ldots,S_t$ such that,
for all collections of $s$ oriented paths of length $t$, there exists an
embedding of vertex-disjoint copies of these paths through the sets,
using one vertex of each $S_i$ on each path.

\begin{definition}[Robust pipeline]\label{def:robust-pipeline}
Pairwise disjoint sets $W_0,\ldots,W_t$, each of order $s$ or $s+1$,
form a \emph{robust pipeline of width $s$ and length $t$} if, for all
$0\leq a\leq b\leq t$ and all choices of equal-sized subsets
$S_i\subseteq W_i$ with $|W_i\setminus S_i|\leq1$ for $a\leq i\leq b$,
the sequence $(S_a,\ldots,S_b)$ is a pipeline of length $b-a$.
\end{definition}

The following lemma combines Proposition~\ref{prop:random-typical} and
Theorem~\ref{thm:lemma8} with \cite[Theorem~6]{HT}.

\begin{lemma}[Robust pipeline construction]\label{lem:robust-pipeline}
There exists an absolute constant $C>0$ such that the following
holds. Let $0<\nu\leq\tau\leq\frac{\eta}{8}<\frac{1}{8}$, let
$n\geq\exp(\frac{C}{\nu^4})$, and let $G$ be an $n$-vertex robust
$(\nu,\tau)$-outexpander with $\delta^0(G)\geq\eta n$. Set
$s=\lceil(\log n)^9\rceil$, and let $U\subseteq V(G)$,
$T\in\mathbb N$, and $J\subseteq\{0,\ldots,T\}$. If
$$
 |V(G)\setminus U|\leq\frac{\nu^3n}{4},~
 3\leq T<\frac{n}{8}\text{,~and~}
 (T+1)s+|J|\leq |U|,
$$
then $U$ contains a robust pipeline
$W_0,\ldots,W_T$ such that
$|W_i|=s+1$ precisely when $i\in J$.
Moreover, let $a,b$ be integers with $0\leq a\leq b\leq T$, and let
$S_i\subseteq W_i$, $a\leq i\leq b$, have common order $w$, with
$|W_i\setminus S_i|\leq1$ for all $a\leq i\leq b$.
The sets $S_a,\ldots,S_b$ are $(\cC,\frac{\nu}{24})$-typical, where
$\cC$ is the container family from Corollary~\ref{lem:containers}.
Set $t=b-a$. If $t\geq2\lceil\frac{16}{\nu}\log_2w\rceil$, then, for
all oriented paths $P_1,\ldots,P_w$ of length $t$, all but at most
three vertices $v\in S_b$ have the following property: there exists
a family of vertex-disjoint copies of $P_1,\ldots,P_w$ through
$S_a,\ldots,S_b$, using one vertex of each layer on each path, in
which the copy of $P_w$ ends at $v$.
\end{lemma}

\begin{proof}
Set $N=|U|$, $K=\frac{96}{\nu^2}$, and $\gamma=\frac{\nu}{48}$.
To apply Proposition~\ref{prop:random-typical} inside $U$, let $\cC_U$
consist of the pairs $(P\cap U,Q\cap U)$, $(P,Q)\in\cC$, together
with $(R,R)$, where $R=U\setminus N_G^\sigma(v)$,
for all $v\in V(G)$ and $\sigma\in\{+,-\}$. For a
$(\cC_U,\gamma)$-typical set $S\subseteq U$, these pairs give
$|N_G^\sigma(v)\cap S|\geq
(\frac{|N_G^\sigma(v)\cap U|}{N}-\gamma)|S|$
for all $v\in V(G)$ and $\sigma\in\{+,-\}$.
For all $R\subseteq V(G)$,
$$
 \left|\frac{|R\cap U|}{N}-\frac{|R|}{n}\right|
 \leq\frac{n-N}{n}\leq\frac{\nu^3}{4}<\gamma.
$$
Thus every $(\cC_U,\gamma)$-typical set in $G[U]$ is
$(\cC,\frac{\nu}{24})$-typical in $G$. Since $N\geq\frac{n}{2}$,
we also have $|\cC_U|\leq n^{48/\nu^2}+2n\leq N^K$ for sufficiently
large $n$.

Choose
$$
 r=\left\lceil\frac{2(K+30)}{\gamma^2}\log n\right\rceil,~
 h_0=\frac{64r}{\nu}\log_2(s+1).
$$
For sufficiently large $C$, Proposition~\ref{prop:random-typical}(ii)
applies to $G[U]$ and $\cC_U$ with either sample order $s$ or $s+1$.
We also have $h_0+1<(\log N)^3$, $h_0\leq s-1$, and
$\frac{32}{\nu}\log_2(s+1)<\frac{n}{2}$.

Choose uniformly an ordered family of pairwise disjoint sets
$W_0,\ldots,W_T\subseteq U$ with $|W_i|=s+1$ for $i\in J$ and
$|W_i|=s$ otherwise. Each $W_i$ is a uniform subset of $U$ of its
prescribed order. Proposition~\ref{prop:random-typical}(ii) therefore
gives, with failure probability at most
$\frac{1}{N}$ per layer, the following property simultaneously for
all $H\subseteq W_i$ with $|H|\leq h_0+1$: if $w'=|W_i\setminus H|$,
then at least $\left(1-\frac{1}{2N}\right)\binom{w'}{u}$ of its
$u$-subsets are $(\cC,\frac{\nu}{24})$-typical in $G$ whenever
$u=0$, $u=w'$, or $r\leq u\leq w'-r$. A union bound over the layers
has failure probability at most $\frac{T+1}{N}<1$, so fix a family for
which this property holds in every layer. In particular, taking $u=w'$
shows that every whole set left after such a deletion is typical.

Fix $a,b$ and sets $S_a,\ldots,S_b$ of common order $w$ as in the
statement, and set $t=b-a$.
We have $w\in\{s-1,s,s+1\}$. We verify the
matched-pair conditions of~\cite{HT} for two consecutive remaining layers.
Delete a further $h$ vertices from each, where
$h\leq\frac{64r}{\nu}\log_2w$, and call the resulting sets $X,Y$.
At most $h+1\leq h_0+1$ vertices have been deleted from each original
layer, so $X,Y$ are typical and Theorem~\ref{thm:lemma8} makes
$(X,Y)$ $\frac{\nu}{4}$-expanding.

Let $u=0$, $u=w-h$, or $r\leq u\leq w-h-r$.
In each of $X$ and $Y$, at least a $1-\frac{1}{N}$ proportion of the
$u$-subsets have both the subset and its complement typical: each of
these two conditions fails for at most a $\frac{1}{2N}$ proportion.
For all such $A\subseteq X$ and $B\subseteq Y$,
Theorem~\ref{thm:lemma8} and Hall's theorem give perfect matchings
directed from $A$ to $B$ and from $Y\setminus B$ to $X\setminus A$;
this is the argument of~\cite[Lemma~5]{HT}.
Since $\frac{1}{N}\leq\frac{2}{n}$, this is precisely the matching
condition for a $(\frac{\nu}{4},r,\frac{2}{n})$-matched pair.
The other two conditions in that definition are
$\frac{64r}{\nu}\log_2w\leq w$ and
$\frac{32}{\nu}\log_2w<\frac{n}{2}$,
both of which follow from the bounds on $h_0$ and $s$ above.

The proof of~\cite[Theorem~6]{HT} uses only these expansion and matching
conditions, so it applies to digraphs as well as oriented graphs.
Since $t\leq T<\frac{n}{4}$, that theorem shows that the remaining
layers form a pipeline.

For the final assertion, suppose
$t\geq2\lceil\frac{16}{\nu}\log_2w\rceil$ and fix oriented paths
$P_1,\ldots,P_w$ of length $t$. The counting argument in that proof,
with $\alpha=\frac{2}{n}$, gives at least $(1-\frac{4t}{n})w!$
bijections $f:[w]\to S_b$ for which there exists a family of
vertex-disjoint copies of $P_1,\ldots,P_w$ through the layers with
the copy of $P_j$ ending at $f(j)$ for all $j\in[w]$.
If $v\in S_b$ cannot be the last vertex of the copy of $P_w$ in
such a family, all $(w-1)!$ bijections with $f(w)=v$ are excluded.
Thus the number of such vertices is at most $\frac{4tw}{n}<4$,
where the strict inequality follows from $(t+1)w\leq n$.
\end{proof}

\subsection{From Hall-type inheritance to universal
Hamilton-connectivity}
\label{sec:HT}

We adapt the construction from \cite[Theorem~12]{HT}, using the sorters
and pipelines supplied by Lemmas~\ref{lem:sorters}
and~\ref{lem:robust-pipeline}, to obtain a spanning oriented path with
prescribed first and last vertices.

\begin{proof}[Proof of Theorem~\ref{thm:main}]
Let $P=p_1p_2\cdots p_n$ be an oriented path and let $x,y\in V(G)$ be
distinct. Write $v_x=p_1$ and $v_y=p_n$; these are the vertices of $P$
which will be mapped to $x$ and $y$, respectively.
We will select disjoint subpaths with common initial and terminal
orientation patterns. Two sorters will realize these patterns, while a
pipeline will contain the equal-length middle paths. The intervening
paths may have different lengths; we join them temporarily into a single
path and embed it using a second pipeline and short connecting paths.
Extra vertices in selected layers of the first pipeline supply these
connections and the connections to the sorters.

Set
$$
 s=\left\lceil(\log n)^9\right\rceil,~
 q=\left\lceil\frac{1}{\nu^3}\right\rceil,~
 d=\left\lceil(\log_2s)^2\right\rceil,~
 \ell=(q+1)d,~
 k=\left\lceil\frac{16}{\nu}\log_2s\right\rceil+2,~
 c=5s(k-1).
$$
Here $s$ is the number of selected subpaths and the width of each sorter
and of the first pipeline. In the construction of
Lemma~\ref{lem:sorters}, there are $d$ sorting stages, each of length
$q+1$, so the paths through each sorter have length $\ell$.
The connecting paths have length $k$, and $c$ counts the internal
vertices of these $5s$ paths:
$s$ of them join the pieces routed through the second pipeline, and
$4s$ of them connect the remaining pieces to the sorters and the first pipeline.
The congruence below will allow us to split the auxiliary path formed
from the intervening pieces into $s+1$ equal-length paths, $s$ connecting
paths of length $k$, and its two end edges.
Choose $t$ with
\begin{equation}\label{eq:direct-congruence}
 c+1\leq t\leq c+s+1
 ~\text{and}~
 n-4+2\ell+5k+t\equiv0\pmod{s+1},
\end{equation}
which is possible since the indicated interval contains $s+1$
consecutive integers, and set $p=2\ell+2k+t$.
Take $K_2$ sufficiently large that Lemmas~\ref{lem:sorters}
and~\ref{lem:robust-pipeline} apply and the estimates below hold. In particular,
\begin{equation}\label{eq:direct-selection}
 \frac{1}{2}\left(\left\lfloor\frac{n-1}{p}\right\rfloor-3\right)
 >s2^{2\ell+2k}.
\end{equation}
Indeed, set $L=\log n$. The definitions give
$\ell=O(\frac{(\log L)^2}{\nu^3})$,
$k=O(\frac{\log L}{\nu})$, and $\log(sp)=O(\log L)$,
with absolute implied constants. Since $L\geq\frac{K_2}{\nu^4}$,
$$
 \frac{(2\ell+2k)\log 2+\log(8sp)}{L}
 =O\left(\frac{(\log K_2)^2}{K_2}\right)
$$
uniformly in $\nu$. Taking $K_2$ sufficiently large gives
$n>8sp2^{2\ell+2k}$, which implies~\eqref{eq:direct-selection}.

Split the first $mp$ edges of $P$, where
$m=\lfloor\frac{n-1}{p}\rfloor$, into consecutive paths
$Z_1,\ldots,Z_m$ of length $p$. Discard $Z_1$ and $Z_m$, and retain
either those with even indices or those with odd indices, whichever class
is larger. The paths in the retained class are pairwise vertex-disjoint,
with at least $p$ edges between consecutive paths, before the first,
and after the last. By~\eqref{eq:direct-selection} and the pigeonhole
principle, the retained class contains paths
$P'_1,\ldots,P'_s$, listed in their order along $P$, whose first and last
$\ell+k$ edges have the same orientation patterns. Let
$Q'_1,\ldots,Q'_{s+1}$ be the intervening paths, including the initial
and final paths, so that
$$
 P=Q'_1P'_1Q'_2P'_2\cdots Q'_sP'_sQ'_{s+1}.
$$
Each $Q'_j$ has length at least $p$. Write $P'_j=A_jB_jP_jC_jD_j$,
where $A_j,D_j$ have length $\ell$, $B_j,C_j$ have length $k$, and
$P_j$ has length $t$. The paths $A_j$ are mutually isomorphic, and
likewise the paths $B_j$, $C_j$, and $D_j$.

Similarly, write
$$
 Q'_1=Q_1F_1,~
 Q'_j=E_jQ_jF_j~(2\leq j\leq s),~
 Q'_{s+1}=E_{s+1}Q_{s+1},
$$
where all $E_j,F_j$ have length $k$. The first vertex of $Q_1$ is
$v_x$, and the last vertex of $Q_{s+1}$ is $v_y$.
Figure~\ref{fig:path-decomposition} shows this decomposition.

\begin{figure}[htbp]
\centering
\begingroup
\def\HTFigurePage{1}
\input{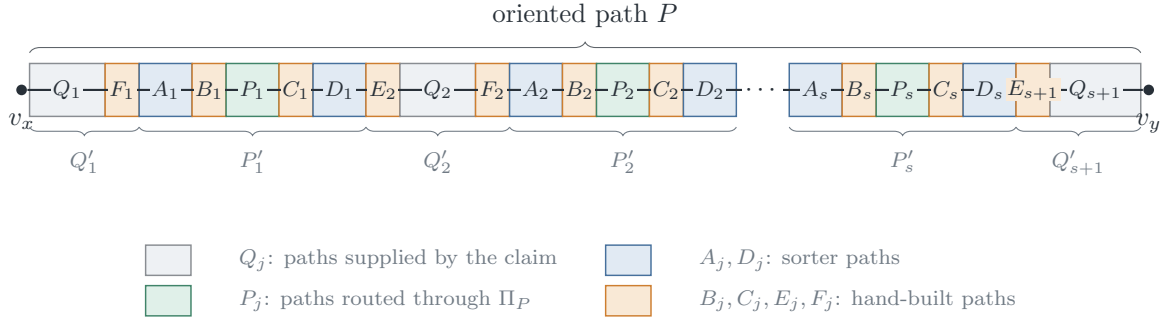}
\endgroup
\caption{The decomposition of the prescribed oriented path $P$.}
\label{fig:path-decomposition}
\end{figure}

For all $j\in[s]$, add an edge $e_j$ of arbitrary orientation between
the last vertex of $Q_j$ and the first vertex of $Q_{j+1}$, and let
$
 Q=Q_1e_1Q_2e_2\cdots Q_se_sQ_{s+1}.
$
This oriented path has length $M=n+s-1-sp-2sk$. We shall construct a
copy of $Q$ and then delete the edges corresponding to
$e_1,\ldots,e_s$.
By~\eqref{eq:direct-congruence}, $u=\frac{M-2-sk}{s+1}$ is an integer,
so $M=2+(s+1)u+sk$. Direct substitution gives
\begin{equation}\label{eq:direct-residual}
 n-2s(\ell+1)-2-(u+1)(s+1)-(t+1)s=c.
\end{equation}
The subtracted terms account for the two sorters, the endpoints $x,y$,
and the two pipelines before selected layers are enlarged; the remaining
$c$ vertices are reserved for the interiors of the connecting paths.

Lemma~\ref{lem:sorters} gives vertex-disjoint sorters
$\Sigma_A,\Sigma_D$ for the common orientation patterns of the $A_j,D_j$,
avoiding $x$ and $y$.

Let $U$ be the set of unused vertices other than $x$ and $y$. The choice
of $K_2$ gives
$|V(G)\setminus U|\leq\frac{\nu^3n}{4}$ and
$3\leq u,t$, $u+t+1<\frac{n}{8}$, and $(\eta-\frac{\nu}{24})s>3$. Set
$$
 T=u+t+1,~
 J_P=\{u+2,\ldots,u+c+1\},~
 J=\{0,\ldots,u\}\cup J_P.
$$
By~\eqref{eq:direct-residual},
$|U|=(T+1)s+|J|$. Apply Lemma~\ref{lem:robust-pipeline} to obtain a
robust pipeline $W_0,\ldots,W_T$ whose $(s+1)$-sets are precisely those
indexed by $J$. By Definition~\ref{def:robust-pipeline},
$\Pi_R=(W_0,\ldots,W_u)$ and $\Pi_P=(W_{u+1},\ldots,W_T)$
are respectively a pipeline of width $s+1$ and a robust pipeline of
width $s$, of lengths $u$ and $t$. For $u+1\leq i\leq T$, choose an
$s$-subset $S_i\subseteq W_i$. By Lemma~\ref{lem:robust-pipeline},
each $S_i$ is $(\cC,\frac{\nu}{24})$-typical, and by
Theorem~\ref{thm:lemma8}, each consecutive pair $(S_i,S_{i+1})$ is
$\frac{\nu}{4}$-expanding. Divide the first $s(k-1)$ indices of $J_P$ into $s$
consecutive blocks of $k-1$ indices.

\begin{claim}
The vertices $x,y$, the pipeline $\Pi_R$, and one vertex from each of
these $s(k-1)$ layers of width $(s+1)$ of $\Pi_P$ contain pairwise vertex-disjoint
copies of $Q_1,\ldots,Q_{s+1}$ in which the copy of $Q_1$ starts at $x$
and the copy of $Q_{s+1}$ ends at $y$.
\end{claim}

\begin{figure}[htbp]
\centering
\begingroup
\def\HTFigurePage{2}
\input{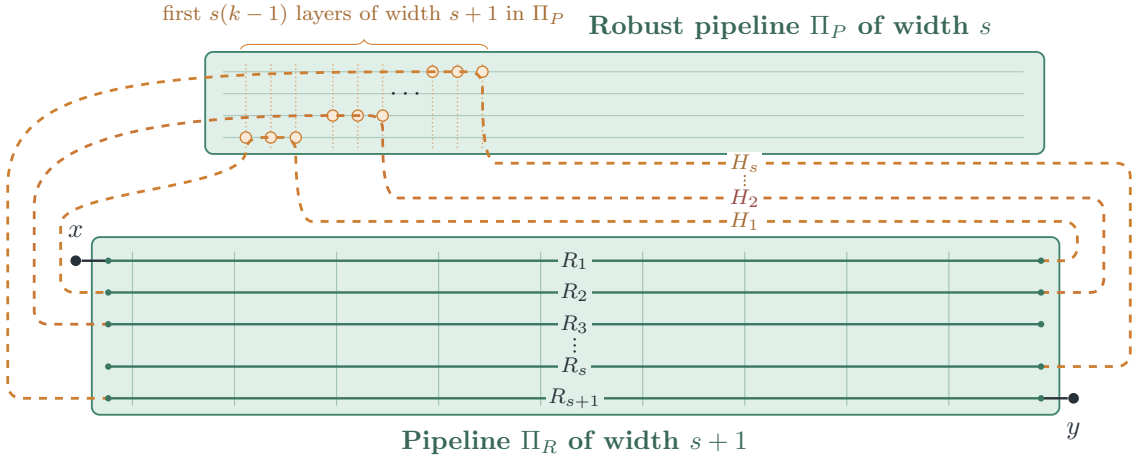}
\endgroup
\caption{The construction of the oriented path $Q$.  The paths $R_i$ are routed through $\Pi_R$, while the paths
$H_i$ use one vertex from each of the first $s(k-1)$ layers of width $s+1$ in
$\Pi_P$.  After deleting the edges $e_1, \dots, e_s$, we obtain the desired paths $Q_1, \dots, Q_{s+1}$.}
\label{fig:claim-routing}
\end{figure}

\begin{claimproof}
In what follows, see Figure \ref{fig:claim-routing}.  Let $e_x$ and $e_y$ be the first and last edges of $Q$, respectively.
Since $M=2+(s+1)u+sk$, write
$$
 Q=e_xR_1H_1R_2H_2\cdots R_sH_sR_{s+1}e_y,
$$
where every $R_i$ has length $u$ and every $H_i$ has length $k$.
Since $W_0,\ldots,W_u$ are typical, we may choose an appropriate
neighbor $a\in W_0$ of $x$ for $e_x$ and embed $R_1$ greedily from $a$,
using one vertex from each layer. Indeed,~\eqref{eq:degree-typical}
gives at least $(\eta-\frac{\nu}{24})(s+1)>0$ choices for the next
vertex. Let $W_i'$ be the $s$-set left in $W_i$ after deleting the
vertex used by this copy of $R_1$, for all $0\leq i\leq u$.
By Lemma~\ref{lem:robust-pipeline}, these sets are typical.

Typicality of $W_u'$ gives at least $(\eta-\frac{\nu}{24})s>3$
appropriate neighbors of $y$ for $e_y$.
Since $u\geq2\lceil\frac{16}{\nu}\log_2s\rceil$,
Lemma~\ref{lem:robust-pipeline} therefore gives a family of
vertex-disjoint copies of $R_2,\ldots,R_{s+1}$ through
$W_0',\ldots,W_u'$ in which $R_{s+1}$ ends at an appropriate
neighbor $b$ of $y$. Adding the edges between
$x$ and $a$ and between $b$ and $y$ gives copies of $e_xR_1$ and
$R_{s+1}e_y$.

For all $i\in[s]$, assign the $i$th block of $k-1$ $(s+1)$-vertex layers to
$H_i$. Choose appropriate neighbors of the end of $R_i$ and the start
of $R_{i+1}$ in the first and last sets $S_j$ of this block; these
neighbors exist by typicality. Since consecutive pairs are
$\frac{\nu}{4}$-expanding and $k-2\geq\frac{16}{\nu}\log_2s$, the
argument of~\cite[Lemma~3]{HT} supplies the intervening path of length
$k-2$. Adding the two end edges gives a copy of $H_i$ using one vertex
from each assigned set.

We have now constructed a copy of $Q$. Deleting the edges corresponding
to $e_1,\ldots,e_s$ leaves pairwise vertex-disjoint copies of
$Q_1,\ldots,Q_{s+1}$. The construction used exactly $s(k-1)$ vertices
of $\Pi_P$, as claimed.
\end{claimproof}

Divide the remaining $4s(k-1)$ indices of $J_P$ into $4s$ consecutive
blocks of $k-1$ indices and construct the remaining paths of length $k$
in the same way, using one block per path. The copies of $F_j$
join the ends of $Q_j$ to distinct input vertices of $\Sigma_A$, and
the copies of $E_{j+1}$ join distinct output vertices of $\Sigma_D$ to
the starts of $Q_{j+1}$. The copies of $B_1$ join the output vertices of
$\Sigma_A$ bijectively to $W_{u+1}$, and the copies of $C_1$ join $W_T$
bijectively to the input vertices of $\Sigma_D$.

Delete the used vertex from each $W_i$ with $i\in J_P$, leaving every
other $W_i$ unchanged. By Definition~\ref{def:robust-pipeline}, the
remaining layers of $\Pi_P$ form a pipeline of width $s$. Embed
$P_1,\ldots,P_s$ through it, and label the copies of $B_1,C_1$ incident
with $P_j$ as $B_j,C_j$, respectively. Since $\Sigma_A,\Sigma_D$ are
sorters, choose their decompositions into copies of $A_j,D_j$ so that,
for all $j\in[s]$, the path
$F_jA_jB_jP_jC_jD_jE_{j+1}$ joins the end of $Q_j$ to the start of
$Q_{j+1}$. Together with $Q_1,\ldots,Q_{s+1}$, these paths form a copy
of $P$ from $x$ to $y$, as shown in Figure~\ref{fig:final-assembly}.

\begin{figure}[!htb]
\centering
\begingroup
\def\HTFigurePage{3}
\input{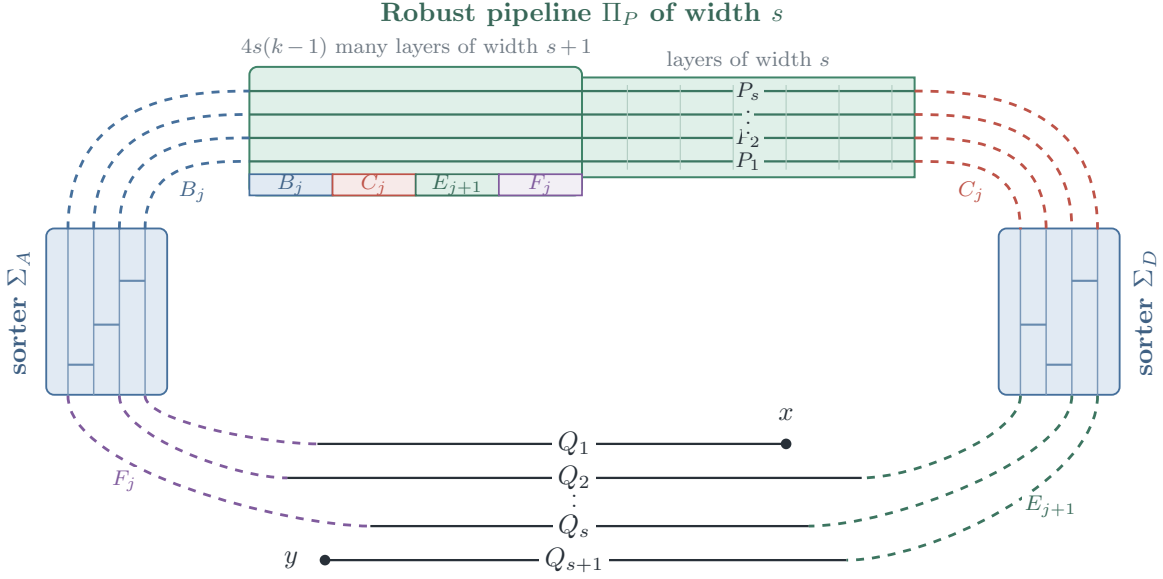}
\endgroup
\caption{The final construction of the oriented path $P$. The dashed paths are built using the layers of width $s+1$ in $\Pi_P$, leaving behind a pipeline of width $s$.  The paths $P_1, \dots, P_s$ are then routed through this remaining pipeline.}
\label{fig:final-assembly}
\end{figure}

This copy uses every vertex of the two sorters and of $W_0,\ldots,W_T$,
as well as $x,y$, and hence spans $G$.

Finally, let $C$ be an orientation of a Hamilton cycle. Choose
any directed edge $yx\in E(G)$, choose an edge $uv$ of $C$ directed from
$u$ to $v$, and delete $uv$. Regard $v$ as the first vertex and $u$ as
the last vertex of the resulting oriented path. This path has a spanning
copy in $G$ from $x$ to $y$. Adding the edge $yx$, which plays the role
of $uv$, gives a copy of $C$.
\end{proof}

\section{Universal linkage in robust expanders}
\label{sec:linkage}

We first prove Theorem~\ref{thm:robust-inherit} and then use it together
with Theorem~\ref{thm:main} to adapt the DeBiasio--Treglown argument and
prove Theorem~\ref{thm:klinked}.

\subsection{Containing sparse pairs}

Lemma~\ref{lem:sparse-threshold} below is the sparse analogue of
Lemma~\ref{lem:threshold-container}: it extends the container algorithm to
sets which span a small number of edges by introducing an additional set
$X$ of controlled size. Container results for sets spanning few edges
appear in the general hypergraph framework of Saxton and
Thomason~\cite{STonline}. The elementary
graph algorithm used here is adapted from Balogh and
Li~\cite[Lemma~2.6]{BL}. The union bound includes the probability that all
vertices represented by a fingerprint lie in the random sample, as in
Blais and Seth~\cite[Section~3]{BlaisSeth}.

\begin{lemma}[Sparse-set container lemma]\label{lem:sparse-threshold}
Let $H$ be a graph on $N$ vertices, let $d>0$, let
$0<\theta<1$, and let
$M\geq0$. There exists a family $\mathcal F$ of ordered pairs of subsets
of $V(H)$ together with a map $C\colon\mathcal F\to2^{V(H)}$
such that
\begin{enumerate}
\item for all $S\subseteq V(H)$ with $e(H[S])\leq M$, there exists
      $(F,X)\in\mathcal F$ such that
      $F\cup X\subseteq S\subseteq C(F,X)$.
\item for all $(F,X)\in\mathcal F$, the sets $F$ and $X$ are disjoint and
      satisfy $|F|\leq \frac{N}{d}$ and
      $|X|\leq\theta N+\frac{M}{\theta d}$.
      Moreover, we have
      $\Delta(H[C(F,X)\setminus X])<d$.
\end{enumerate}
Consequently,
$$
 |\mathcal F|
 \leq
 \left(\sum_{i=0}^{\lfloor N/d\rfloor}\binom Ni\right)
 \left(\sum_{j=0}^{\left\lfloor\min\{N,\theta N+M/(\theta d)\}\right\rfloor}
             \binom Nj\right).
$$
Every container also satisfies
\begin{equation}\label{eq:sparse-container-edges}
 e\bigl(H[C(F,X)]\bigr)
 <\frac{Nd}{2}
   +N\left(\theta N+\frac{M}{\theta d}\right).
\end{equation}
\end{lemma}

\begin{proof}
Fix a linear ordering $\prec$ of $V(H)$. For all sets
$S\subseteq V(H)$ satisfying $e(H[S])\leq M$, we define sets
$U_i,F_i,X_i$ recursively. Begin with $U_0=V(H)$ and
$F_0=X_0=\emptyset$.
Suppose that $U_i,F_i,X_i$ have been defined. If
$\Delta(H[U_i])<d$, stop and set $U_S=U_i$, $F_S=F_i$, and $X_S=X_i$.
Otherwise, let $v_i$ be the $\prec$-first vertex of $U_i$ whose degree in
$H[U_i]$ is at least $d$. Update the three sets according to the
following three cases:  If $v_i\notin S$, delete only $v_i$. If $v_i\in S$ and
$|N_{H[U_i]}(v_i)\cap S|<\theta d$,
add $v_i$ to $F_i$, add $N_{H[U_i]}(v_i)\cap S$ to $X_i$, and delete
$v_i$ together with all its neighbors in $U_i$. Finally, if $v_i\in S$
and $|N_{H[U_i]}(v_i)\cap S|\geq\theta d$,
add $v_i$ to $X_i$ and delete only $v_i$. All sets not mentioned in
a case are left unchanged.  We summarize these cases in the following table:
$$
\setlength{\arraycolsep}{4pt}
\begin{array}{c|c|c|c}
 \text{condition} & U_{i+1} & F_{i+1} & X_{i+1}\\
 \hline
 v_i\notin S
   &U_i\setminus\{v_i\}
   &F_i
   &X_i\\[2mm]
 \substack{v_i\in S,\\
 |N_{H[U_i]}(v_i)\cap S|<\theta d}
   &U_i\setminus\bigl(\{v_i\}\cup N(v_i)\bigr)
   &F_i\cup\{v_i\}
   &X_i\cup\bigl(N_{H[U_i]}(v_i)\cap S\bigr)\\[3mm]
 \substack{v_i\in S,\\
 |N_{H[U_i]}(v_i)\cap S|\geq\theta d}
   &U_i\setminus\{v_i\}
   &F_i
   &X_i\cup\{v_i\}
\end{array}
$$
At least one vertex is deleted at every step, so the algorithm terminates. Note that $F_S$ is independent, since whenever a vertex is added to $F_S$, all of its current neighbors are deleted.

Whenever a vertex is added to $F_S$, at least $d$ vertices of the
current set $U_i$ are deleted as its neighbors. The sets deleted at different
such steps are disjoint, and hence $|F_S|\leq \frac{N}{d}$.
A member of $F_S$ contributes fewer than $\theta d$ neighbors to $X_S$,
so these additions contribute at most
$\theta d|F_S|\leq\theta N$ vertices. Let $X_S^{\rm sel}$ be the set of
vertices added to $X_S$ in the third case. For all $v\in X_S^{\rm sel}$,
let $E_v$ be the set of edges from $v$ to vertices of $S$ which belong to
the current set $U_i$ when $v$ is selected. We have $|E_v|\geq\theta d$. Moreover,
the sets $E_v$, for $v\in X_S^{\rm sel}$, are pairwise disjoint. Indeed,
if $v$ is selected before $w$, then $v$ has already been deleted when $w$
is selected, so the edge $vw$ cannot belong to $E_w$. It follows that
$$
 \theta d|X_S^{\rm sel}|
 \leq\sum_{v\in X_S^{\rm sel}}|E_v|
 \leq e(H[S])
 \leq M.
$$
Thus $|X_S^{\rm sel}|\leq \frac{M}{\theta d}$, and consequently $|X_S|\leq\theta N+\frac{M}{\theta d}$.

As in the proof of Lemma~\ref{lem:threshold-container}, rerunning the
algorithm with $F_S$ in place of $S$ gives the same current sets.
Indeed, if the current sets agree, both runs select the same vertex $v_i$.
If $v_i\in F_S$, the original run used the second case, and so does the
second run because $F_S$ is independent; both delete $v_i$ and its current
neighbors. Otherwise, the original run used the first or third case,
and both delete only $v_i$. Induction therefore shows that the terminal
set $U_S$ is determined by $F_S$; write it as $U_{F_S}$.

Let $\mathcal F$ be the family of fingerprints
$(F_S,X_S)$ arising from sets $S$ with
$e(H[S])\leq M$, and define
$C(F,X)=F\cup X\cup U_F$.
The two members of a fingerprint are disjoint and are contained in $S$.
A vertex of $S$ which is selected belongs to $F_S$ or $X_S$, while a
vertex of $S$ deleted as a neighbor of a member of $F_S$ is placed in
$X_S$. Every other vertex of $S$ remains in the
terminal set. Thus $F_S\cup X_S\subseteq S\subseteq C(F_S,X_S)$.
Every vertex of $F$ is isolated in $H[C(F,X)\setminus X]$: the set $F$
is independent, and all remaining neighbors of a fingerprint vertex are
deleted when that vertex is selected. Since
$C(F,X)\setminus X=F\cup U_F$, while the stopping rule gives
$\Delta(H[U_F])<d$, this proves the maximum-degree assertion. Counting the possible choices
for the two members of a fingerprint gives the stated bound on
$|\mathcal F|$.

Finally, the maximum-degree assertion gives
$e(H[C(F,X)\setminus X])<\frac{Nd}{2}$. Every remaining edge of
$H[C(F,X)]$ is incident with $X$, so there are at most $N|X|$ such edges.
The bound on $|X|$ now gives
\eqref{eq:sparse-container-edges}.
\end{proof}

\subsection{Inheritance of robust expansion}\label{sec:inherit_robust}

\begin{proof}[Proof of Theorem~\ref{thm:robust-inherit}]
A failure of robust expansion in the sample gives a sparse ordered pair,
which we place in one of the containers from
Lemma~\ref{lem:sparse-threshold}. We include the probability that its
fingerprint lies in the sample when summing over containers; this is
what allows the estimate to remain useful for small samples.

Choose $K_1$ sufficiently large, and set
$\zeta=(\frac{16}{\nu^2}+1)\sqrt\rho$.
We may assume that, for all $0<\nu<\frac{1}{2}$,
\begin{equation}\label{eq:log-inheritance-parameters}
 \rho+4\zeta\leq\frac{\nu}{4}
 ~\text{and}~
 2\zeta\left(1+\log\frac{9e}{2\zeta}\right)
 \leq\frac{\nu^2}{64}.
\end{equation}
Indeed, $\zeta=\frac{(16+\nu^2)\nu^2}{\sqrt{K_1}\log(K_1/\nu)}$, and, for large
$K_1$, we have $1+\log(\frac{9e}{2\zeta})\leq4\log(\frac{K_1}{\nu})$. The second left-hand
side is therefore at most $\frac{136\nu^2}{\sqrt{K_1}}$, and the first
inequality follows from the same expression for $\zeta$.

Apply Lemma~\ref{lem:sparse-threshold} to the bipartite representation
$H_G$ with parameters
$$
 d=\frac{\nu^2n}{8},~
 M=\rho m^2,~
 \theta=\frac{8\sqrt\rho\,m}{\nu^2n}.
$$
The choice of $K_1$ gives $0<\theta<1$. For all fingerprints $(F,X)\in\mathcal F$, let
$P=\{v:v_L\in C(F,X)\}$ and $Q=\{v:v_R\in C(F,X)\}$.
Since $H_G$ has $2n$ vertices, the fingerprint bounds give
$|F|\leq\frac{16}{\nu^2}\leq\zeta m$ and $|X|\leq\zeta m$.
Since \eqref{eq:log-inheritance-parameters} gives
$2\zeta\leq\frac{\nu^2}{64}$ and $m\leq n$,
\eqref{eq:sparse-container-edges} implies
$$
 e_G(P,Q)<\frac{\nu^2n^2}{8}+2\zeta mn
 <\frac{\nu^2n^2}{4}.
$$

If $G[W]$ is not a robust $(\rho,2\tau)$-outexpander, choose
$A\subseteq W$ with $2\tau m\leq|A|\leq(1-2\tau)m$ and
$|\RN_{\rho,G[W]}(A)|<|A|+\rho m$. Set
$B=W\setminus\RN_{\rho,G[W]}(A)$. We have
$$
 e_G(A,B)<\rho m^2,~
 |B|>(2\tau-\rho)m,~
 |A|+|B|>(1-\rho)m.
$$
The set $A_L\cup B_R$ spans $e_G(A,B)<M$ edges in $H_G$, so there exists
$(F,X)\in\mathcal F$ such that
$F\cup X\subseteq A_L\cup B_R\subseteq C(F,X)$.
Thus $A\subseteq W\cap P$ and $B\subseteq W\cap Q$, and every vertex
represented in $F\cup X$ lies in $W$.

Fix a fingerprint $(F,X)$, let $\pi(v_L)=\pi(v_R)=v$, and set
$T=\pi(F\cup X)$ and $t=|T|\leq2\zeta m$. Conditional on
$T\subseteq W$, the set $W\setminus T$ is a uniformly random
$(m-t)$-subset of $V(G)\setminus T$. For either $U=P$ or $U=Q$,
$$
 \mathbb E(|W\cap U|\mid T\subseteq W)
 =|T\cap U|+\frac{m-t}{n-t}|U\setminus T|
 \leq t+\frac{m}{n}|U|.
$$
Lemma~\ref{lem:chernoff-hoeffding} therefore shows that, except with
conditional probability at most $2e^{-\nu^2m/32}$, both inequalities
$$
 |W\cap P|\leq\frac{m}{n}|P|+t+\frac{\nu m}{8}
 ~\text{and}~
 |W\cap Q|\leq\frac{m}{n}|Q|+t+\frac{\nu m}{8}
$$
hold. If such sets $A$ and $B$ lie in $P$ and $Q$, respectively,
these inequalities and \eqref{eq:log-inheritance-parameters} imply
\begin{align*}
|P|\geq(2\tau-2\zeta-\frac{\nu}{8})n\geq\tau n,~ |Q|>(2\tau-\rho-2\zeta-\frac{\nu}{8})n\geq\tau n\text{,~and}\\
 |P|+|Q|>(1-\rho-4\zeta-\frac{\nu}{4})n\geq(1-\frac{\nu}{2})n.
\end{align*}
Lemma~\ref{lem:density} would then give $e_G(P,Q)\geq\frac{\nu^2n^2}{4}$,
contrary to $e_G(P,Q)<\frac{\nu^2n^2}{4}$. Hence, conditional on $T\subseteq W$,
the probability that this container contains sets $A_L\cup B_R$ as
above is at most $2e^{-\nu^2m/32}$.

As in Blais and Seth~\cite[Section~3]{BlaisSeth}, we include the factor
$\Pr(T\subseteq W)$ in the union bound. For a fixed $t$-set $T$, at
most $9^t$ fingerprints satisfy $\pi(F\cup X)=T$, since each of the two
copies of every vertex in $T$ can lie in $F$, in $X$, or in neither set.
Since
$\Pr(T\subseteq W)\leq(\frac{m}{n})^t$, we obtain
\begin{align*}
 \sum_{(F,X)\in\mathcal F}\Pr(\pi(F\cup X)\subseteq W)
 &\leq\sum_{t=0}^{\lfloor2\zeta m\rfloor}
       \binom nt9^t\left(\frac{m}{n}\right)^t\\
 &\leq(2\zeta m+1)
       \left(\frac{9e}{2\zeta}\right)^{2\zeta m}
 \leq\exp(\frac{\nu^2m}{64}).
\end{align*}
The second inequality uses the monotonicity of
$x\log(\frac{9em}{x})$ for $0<x\leq2\zeta m$, and the last follows
from \eqref{eq:log-inheritance-parameters} and
$\log(2\zeta m+1)\leq2\zeta m$. Every failure of robust expansion
gives one of the fingerprints considered above, so
$$
 \Pr(G[W]\text{ is not a robust $(\rho,2\tau)$-outexpander})
 \leq2e^{-\nu^2m/32}e^{\nu^2m/64}
 =2e^{-\nu^2m/64}.
$$

Finally, for all $v\in V(G)$ and $\sigma\in\{+,-\}$,
Lemma~\ref{lem:chernoff-hoeffding} gives
$\Pr(|N^\sigma(v)\cap W|<\frac{\eta m}{2})\leq e^{-\eta^2m/2}$.
A union bound over the $2n$ choices proves the second estimate.
\end{proof}

\begin{remark}\label{rem:hall-growing-samples}
For fixed $\nu,\tau,\eta$, the pair $(S,S')$ in
Theorem~\ref{thm:hall-inherit} is $\frac{\nu}{4}$-expanding with
probability $1-o(1)$ whenever $m\to\infty$. Indeed, retain the fingerprints
$\mathcal F$ used in the proof of Corollary~\ref{lem:containers}, and set
$h=\lfloor\frac{16}{\nu^2}\rfloor$. For all $F\in\mathcal F$, let
$E_F=\{F\subseteq S_L\cup S'_R\}$. Since $|F|\leq h$ and
$|S_L\cup S'_R|=2m$,
$$
 \sum_{F\in\mathcal F}\Pr(E_F)
 \leq\sum_{j=0}^h\binom{2m}{j}
 \leq(h+1)(2m)^h.
$$
For a fingerprint with $\Pr(E_F)>0$, conditioning on $E_F$ fixes at most
$h$ vertices in prescribed parts. The means of $|S\cap P(F)|$ and
$|S'\cap Q(F)|$ consequently increase by at most $h$. If
$m\geq\frac{48h}{\nu}$, Lemma~\ref{lem:chernoff-hoeffding} shows that the
two corresponding bounds in~\eqref{eq:container-typical}, with
$\gamma=\frac{\nu}{24}$, fail with conditional probability at most
$2e^{-\nu^2m/1152}$.
The probability that any sampled vertex has fewer than
$(\eta-\frac{\nu}{24})m$ inneighbors or outneighbors in the other sample
is at most $4m e^{-\nu^2m/288}$, by summing the degree bounds conditional
on membership in the sample.
Applying the argument in the proof of Theorem~\ref{thm:lemma8} and
summing over fingerprints for both orders of $S,S'$ therefore bounds
the failure probability by
$$
 4(h+1)(2m)^h e^{-\nu^2m/1152}
 +4m e^{-\nu^2m/288}=o(1).
$$
\end{remark}

The following corollary is the regularity-free replacement for the
splitting lemma~\cite[Lemma~5.2]{DT}. We include the prescribed
exceptional set because, in the linkage application, it consists of the
endvertices of the paths still to be embedded.

\begin{corollary}
\label{cor:random-linear-partition}
Let $K_1$ be the absolute constant in
Theorem~\ref{thm:robust-inherit}. Let
$0<\nu\leq\tau\leq\frac{\eta}{8}<\frac{1}{8}$, set
$\rho=\frac{(\nu/2)^8}{K_1\log^2(2K_1/\nu)}$,
and let $G$ be an $n$-vertex robust $(\nu,\tau)$-outexpander satisfying
$\dmin(G)\geq\eta n$. Let $W_0\subseteq V(G)$ have order
$m_0\leq\frac{\nu n}{16}$, and suppose that
$m_0+m_1+\cdots+m_t=n$, where
$$
 m_i\geq\max\left\{
 \frac{16K_1}{\nu^4}\log(\frac{2K_1}{\nu}),
 \frac{256}{\nu^2}(4\log n+\log 4),
 \frac{8}{\eta^2}(5\log n+\log 4)
 \right\}
 ~\text{for all}~i\in [t].
$$
Choose an ordered partition
$V(G)\setminus W_0=W_1\cup\cdots\cup W_t$ uniformly subject to
$|W_i|=m_i$. With probability at least $1-\frac{1}{n^2}$, for all $i\in[t]$,
the digraph $G[W_i]$ is a robust $(\rho,4\tau)$-outexpander and, for all
$v\in V(G)$, we have
$|N_G^\pm(v)\cap W_i|\geq\frac{\eta m_i}{4}$.
\end{corollary}

\begin{proof}
Let $F=G-W_0$ and $N=|F|$. Lemma~\ref{lem:bounded-deletion} shows that
$F$ is a robust $(\frac{\nu}{2},2\tau)$-outexpander with
$\delta^0(F)\geq\frac{\eta N}{2}$.

Each $W_i$ is a uniformly random $m_i$-subset of $V(F)$.
Theorem~\ref{thm:robust-inherit}, applied to $F$ with parameters
$\frac{\nu}{2},2\tau,$ and $\frac{\eta}{2}$, together with a union bound over $i\in[t]$,
shows that, with failure probability at most $\frac{1}{2n^2}$, every $G[W_i]$
has the required robust expansion and every vertex of $F$ has at least
$\frac{\eta m_i}{4}$ inneighbors and outneighbors in $W_i$.

For $z\in W_0$, we have
$|N_G^\pm(z)\cap V(F)|\geq\eta n-|W_0|\geq\frac{\eta N}{2}$.
The Chernoff--Hoeffding bound therefore gives
$\Pr(|N_G^\pm(z)\cap W_i|<\frac{\eta m_i}{4})\leq e^{-\eta^2m_i/8}$.
The stated lower bound on $m_i$ and a union bound over the at most
$2m_0t\leq2n^2$ choices of $z,i,$ and the sign show that these endpoint
degree conditions fail with probability at most $\frac{1}{2n^2}$.
\end{proof}

\subsection{From robust inheritance to universal linkage}

We now follow the proof of~\cite[Theorem~3.5]{DT}, replacing their
regularity-based splitting lemma by
Corollary~\ref{cor:random-linear-partition} and Taylor's theorem by
Theorem~\ref{thm:main}.

\begin{proof}[Proof of Theorem~\ref{thm:klinked}]
Let $K_1$ and $K_2$ be the absolute constants in
Theorems~\ref{thm:robust-inherit} and~\ref{thm:main}, and set
$a=\frac{\beta^2\nu}{64}$. Since $\ell_1+\cdots+\ell_k=n$ and
$\ell_i\geq\frac{1}{\beta}$, we also have $k\leq\beta n$.

First embed successively every $Q_i$ with $\ell_i\leq an$. When embedding
$Q_i$, delete the vertices in the copies already constructed and all
prescribed endpoints other than $u_i,v_i$. The number of deleted vertices
is at most $kan+2k\leq\frac{\nu n}{64}+2\beta n\leq\frac{\nu n}{16}$.
Let $F$ be the remaining digraph. By Lemma~\ref{lem:bounded-deletion},
$F$ is a robust $(\frac{\nu}{2},2\tau)$-outexpander with
$\delta^0(F)\geq\frac{\eta|F|}{2}$.
Since $\beta\leq\frac{\nu^3}{64}$ and $|F|\geq(1-\frac{\nu}{16})n$,
$$
 \frac{16}{\nu^3}
 \leq\ell_i-1
 \leq an
 \leq\frac{\nu^3}{32}|F|.
$$
Lemma~\ref{lem:fixed-endpoint-paths},
with parameters $\frac{\nu}{2},2\tau,$ and $\frac{\eta}{2}$, therefore gives
the desired copy from $u_i$ to $v_i$. Repeating this embeds all paths with
$\ell_i\leq an$, with pairwise disjoint copies avoiding the endpoints of
the remaining paths.

Delete these short paths, and call the remaining graph $G'$ and its order
$N$. Their union has order at most $kan\leq\frac{\nu n}{64}$, so
$N\geq(1-\frac{\nu}{64})n$. Lemma~\ref{lem:bounded-deletion} shows that
$G'$ is a robust
$(\frac{\nu}{2},2\tau)$-outexpander with
$\delta^0(G')\geq\frac{\eta N}{2}$. Relabel the paths still to be embedded
as $Q_1,\ldots,Q_t$, and set $W_0=\{u_1,\ldots,u_t,v_1,\ldots,v_t\}$ and $m_i=\ell_i-2$.
Since $t\leq k\leq\beta n$, we have
$|W_0|=2t\leq2\beta n\leq\frac{\nu N}{32}$.
Each remaining path satisfies
$$
 \ell_i\geq\max\left\{an,\frac{1}{\beta}\right\}
 \geq\left(\frac{\nu n}{64}\right)^{1/3}.
$$
Consequently, $m_i=\ell_i-2\geq n^{1/4}$ for sufficiently large $K_3$.
Set $\rho=\frac{(\nu/4)^8}{K_1\log^2(4K_1/\nu)}$.
There exists an absolute constant $D$ such that
$\frac{1}{\rho^4}\leq\frac{D}{\nu^{40}}$. Since
$\log m_i\geq\frac{\log n}{4}\geq\frac{K_3}{4\nu^{40}}$, taking $K_3$
sufficiently large ensures that every $m_i$ meets
the size hypotheses of Corollary~\ref{cor:random-linear-partition} with
parameters $\frac{\nu}{2},2\tau,$ and $\frac{\eta}{2}$, and also that
$\rho m_i\geq8$ and $m_i\geq\exp(\frac{16K_2}{\rho^4})$.

Apply Corollary~\ref{cor:random-linear-partition} to $G'$ with parameters
$\frac{\nu}{2},2\tau,$ and $\frac{\eta}{2}$, prescribed exceptional set
$W_0$, and part sizes $m_1,\ldots,m_t$. We obtain a partition
$V(G')=W_0\cup W_1\cup\cdots\cup W_t$ such that, for all $i\in[t]$,
$G[W_i]$ is a robust
$(\rho,8\tau)$-outexpander and
$|N_G^\pm(v)\cap W_i|\geq\frac{\eta m_i}{8}$ for all $v\in V(G')$.

For all $i\in[t]$, set $U_i=W_i\cup\{u_i,v_i\}$. The degree bound
applies to both endpoints as well as to every vertex of $W_i$, so
$\delta^0(G[U_i])\geq\frac{\eta m_i}{8}\geq\frac{\eta}{16}|U_i|$.
To check robust expansion, let $S\subseteq U_i$ satisfy
$16\tau|U_i|\leq|S|\leq(1-16\tau)|U_i|$ and
set $S_0=S\cap W_i$. Since $\rho m_i\geq8$ and $\rho\leq\tau$, we have
$8\tau m_i\leq|S_0|\leq(1-8\tau)m_i$.
Every vertex of $\RN_{\rho,G[W_i]}(S_0)$ has at least
$\rho m_i\geq\frac{\rho}{2}|U_i|$ inneighbors in $S$, and hence
\begin{align*}
 |RN^+_{\rho/2,G[U_i]}(S)|
 \geq |S_0|+\rho m_i
 \geq |S|-2+\rho m_i
 \geq |S|+\frac{\rho}{2}|U_i|.
\end{align*}
Thus $G[U_i]$ is a robust $(\frac{\rho}{2},16\tau)$-outexpander.

Since $\frac{\rho}{2}\leq16\tau\leq\frac{\eta}{128}$ and
$|U_i|\geq\exp(\frac{16K_2}{\rho^4})$, Theorem~\ref{thm:main}, with
parameters $\frac{\rho}{2},16\tau,$ and $\frac{\eta}{16}$, gives a spanning
copy $P_i$ of $Q_i$ from $u_i$ to $v_i$ in each $G[U_i]$.
Together with the short paths embedded earlier, these paths are pairwise
disjoint and cover $V(G)$.
\end{proof}

\section{Concluding remarks}
\label{sec:concluding}

Cai, Wang, Wang, and Zhang~\cite{CWWZ} recently determined the exact
Ore-type degree threshold for an anti-directed Hamilton cycle in an oriented
graph. Their proof uses Taylor's theorem in the robust-expander case and
its fixed-endpoint Hamilton-path version in the final step of the
non-expander case. Replacing both inputs with
Theorem~\ref{thm:main} gives a
regularity-free proof of their result.

Wang, Wang, and Zhang~\cite{WWZ} recently determined the exact minimum
semidegree threshold for containing every orientation of a Hamilton cycle,
together with a corresponding pancyclicity result. Theorem~\ref{thm:main}
removes the use of regularity from the robust-expander part of their proof.
Their treatment of the extremal cases still uses the diregularity lemma and
the blow-up lemma. It remains an interesting problem to find a completely
regularity-free proof of their results.

Lang~\cite[Section~9]{Lang} asks whether his inheritance lemma for
uniformly dense hypergraphs can be proved without the weak hypergraph
regularity lemma. It seems likely that the general hypergraph
partition-testing theorem of Fischer, Matsliah, and Shapira~\cite{FMS}
provides an affirmative answer, although we do not work out this
deduction here. It would also be interesting to obtain a direct proof
using containers.

A related inheritance problem arises in sparse undirected graphs.
For an undirected graph $G$ on $n$ vertices, let $e_G(X,Y)$ count
ordered adjacent pairs in $X\times Y$. Following~\cite{MST}, $G$ is
\emph{$(p,\beta)$-sparse} if
$e_G(X,Y)\leq p|X||Y|+\beta\sqrt{|X||Y|}$ for all
$X,Y\subseteq V(G)$, and is a \emph{$\gamma p$-expander} if
$e_G(X,Y)\geq\gamma p|X||Y|$ for all partitions $V(G)=X\cup Y$.
Mousset,
\v{S}kori\'c, and Truji\'c~\cite[Lemma~3.5]{MST} proved that if $G$ is a
$(p,\beta)$-sparse $\gamma p$-expander and $R$ is a uniformly random
$r$-set satisfying $rp\geq C\max\{\log n,\beta\}$, then with probability
at least $1-\frac{1}{n}$, the graph $G[R]$ is a $\gamma_1p$-expander. Their
proof uses the sparse regularity lemma and gives
$\gamma_1=\Theta(\gamma^4)$; they ask whether one can take
$\gamma_1=\Omega(\gamma)$. It would be interesting to determine whether
the sparse-pair container method developed here gives a regularity-free
proof of their inheritance lemma, possibly with improved dependence on
$\gamma$. This would not by itself give a regularity-free proof of their
cycle-cover theorem, since their Embedding Lemma uses sparse regularity
separately.

\section*{Acknowledgements}

Thank you to Bob Krueger for a discussion about the container arguments used in this paper, and thank you to Andrew Treglown for comments on an earlier draft of this paper.

This paper completes a long-standing goal of the author, going back to the time of \cite{DM, DKMOT}. The renewed interest in the problem was inspired by three unrelated interactions, all of which happened within a two-week period in the summer of 2026. First, I was on the master's thesis committee of William Joseph who gave a talk about his work with Tao Jiang involving the container method in June 2026. Second, I had a conversation with Alp M\"uyesser and Arjun Ranganathan at the British Combinatorial Conference in July 2026 about the H\"aggkvist-Thomason paper. Third, a referee of \cite{DT} mentioned their surprise that regularity was needed to prove \cite[Lemma 5.2]{DT} and suggested that we add a comment about why this was the case. I had long wondered if containers could be used as the ingredient H\"aggkvist and Thomason needed to complete their proof as they envisioned it. Reading their proof again primed by the three interactions described above, I was struck by how the H\"aggkvist-Thomason lemma seemed like a proto-container argument.

\section*{AI declaration}

The proofs of Theorems~\ref{thm:hall-inherit}
and~\ref{thm:robust-inherit} were conceived of and written by the author.
ChatGPT-5.6 Sol was used throughout Section \ref{sec:hamilton} and Section \ref{sec:linkage} to check the parameters as we made the appropriate modifications to the arguments of H\"aggkvist--Thomason and DeBiasio--Treglown. ChatGPT was also used to turn my hand-drawn figures into TikZ figures.

Throughout the writing of this paper I was vaguely aware of some connection between my result and the literature on property testing and the weak regularity lemma of Frieze and Kannan.  However, ChatGPT-5.6 couldn't help me prove that any of the existing methods directly implied Theorem~\ref{thm:hall-inherit} (which is the main result needed in the H\"aggkvist--Thomason argument).  However, during the final stages of writing, I asked ChatGPT-6 Astra to carry out a deeper literature search to make sure we weren't missing anything. During this process it identified a number of related results, which led to an expansion in the scope of the paper.  In particular, the original proof of Theorem~\ref{thm:hall-inherit} only applied to sets of at least logarithmic size, and the original proof of Theorem~\ref{thm:robust-inherit} only applied to linear sized sets.  ChatGPT-6 Astra found the paper of Blais and Seth which used a weighted union bound calculation.  We were able to easily adapt this to our setting, thereby extending both results to hold for sets of arbitrary growing size. Additionally, it found the partition-testing and cut-norm sampling results from which the alternate proofs (with slightly different quantitative conclusions) of Theorems~\ref{thm:hall-inherit} and~\ref{thm:robust-inherit} in the appendix were obtained.  The appendix was almost entirely written by ChatGPT-6 Astra; I only made small edits.  While I was initially disappointed to learn that my proof had lost some of its novelty, I am hopeful that highlighting these different approaches will be of value to the relevant research communities.

\newpage

\appendix
\section{Alternate proofs of expansion inheritance}
\label{app:alternate-inheritance}

We give two alternate proofs of the inheritance conclusions of
Theorems~\ref{thm:hall-inherit} and~\ref{thm:robust-inherit}, with different
quantitative bounds. The first uses the Frieze--Kannan weak regularity
lemma together with a cut-norm sampling theorem of Alon, Fernandez de la
Vega, Kannan, and Karpinski. The second uses the partition-testing theorem
of Fischer, Matsliah, and Shapira, whose proof does not use regularity.
Both arguments give the following result, whose parts correspond to
Theorems~\ref{thm:hall-inherit} and~\ref{thm:robust-inherit}, respectively.

\begin{theorem}\label{thm:alternate-inheritance}
For all $0<\nu<\frac{1}{2}$, there exist $m_0=m_0(\nu)$ and $c=c(\nu)>0$
such that the following holds. Let
$\nu\leq\tau\leq\frac{\eta}{2}<\frac{1}{2}$, let $G$ be an $n$-vertex
robust $(\nu,\tau)$-outexpander with $\dmin(G)\geq\eta n$, and let
$m\geq m_0$.
\begin{enumerate}[label=\textup{(\roman*)}]
\item If $\tau\leq\frac{\eta}{8}$, $2m\leq n$, and $(S,S')$ is chosen
uniformly from the ordered pairs of disjoint $m$-subsets of $V(G)$,
then $(S,S')$ is $\frac{\nu}{4}$-expanding with probability at least
$1-e^{-cm}$.
\item If $m\leq n$ and $W$ is a uniformly random $m$-subset of $V(G)$,
then $G[W]$ is a robust $(\frac{\nu^2}{64},2\tau)$-outexpander with
probability at least $1-e^{-cm}$.
\end{enumerate}
In~\textup{(ii)}, the degree estimate in
Theorem~\ref{thm:robust-inherit} also holds.
\end{theorem}

Part~(i) gives the same $\frac{\nu}{4}$-expansion as
Theorem~\ref{thm:hall-inherit}, with failure probability $e^{-cm}$ in
place of $8n^{48/\nu^2}e^{-\nu^2m/288}$, provided $m\geq m_0(\nu)$.
For fixed parameters, this gives probability $1-o(1)$ whenever
$m\to\infty$. The container argument in
Remark~\ref{rem:hall-growing-samples} gives the same conclusion, with an
explicit failure bound independent of $n$.
Part~(ii) replaces the robust expansion parameter
$\rho=\frac{\nu^8}{K_1\log^2(K_1/\nu)}$ in
Theorem~\ref{thm:robust-inherit} by $\frac{\nu^2}{64}$.
That theorem already allows $m$ to grow arbitrarily slowly.
The difference is that its sample threshold and exponential rate are
explicit, whereas we leave $m_0(\nu)$ and $c(\nu)$ unquantified here.
The simultaneous degree estimate for all vertices of $G$ still has
failure probability at most $2n e^{-\eta^2m/2}$.

The following proposition shows that the lower bound on edges between
large sets in Lemma~\ref{lem:density} is inherited by random samples,
after increasing the relative size thresholds and decreasing the
normalized edge bound. Failure of robust expansion gives
a sparse pair violating this bound. For Hall-type expansion, a large
witnessing set gives an empty pair violating the bound, while small
sets are handled by degree concentration. 

\begin{proposition}\label{prop:alternate-sparse-pairs}
For all $0<\nu<\frac{1}{2}$, there exist $m_0=m_0(\nu)$ and $c=c(\nu)>0$
such that the following holds. Let $\nu\leq\tau<\frac{1}{2}$, and suppose
that an $n$-vertex digraph $G$ satisfies the following condition for all
$P,Q\subseteq V(G)$:
\begin{equation}\label{eq:alternate-host-density}
 e_G(P,Q)\geq\frac{\nu^2n^2}{4}
 ~\text{whenever}~
 |P|,|Q|\geq\tau n,~
 |P|+|Q|\geq\left(1-\frac{\nu}{2}\right)n.
\end{equation}
Let $m\geq m_0$. Either let $U=V$ be a uniformly random $m$-subset of
$V(G)$, or choose $(U,V)$ uniformly from the ordered pairs of disjoint
$m$-subsets, with $m\leq n$ or $2m\leq n$, respectively. With probability
at least $1-e^{-cm}$, all $A\subseteq U$ and $B\subseteq V$ satisfying
\begin{equation}\label{eq:alternate-sample-sizes}
 |A|,|B|\geq\frac{3\tau m}{2},~
 |A|+|B|\geq\left(1-\frac{\nu}{4}\right)m
\end{equation}
also satisfy $e_G(A,B)>\frac{\nu^2m^2}{64}$.
\end{proposition}

We first deduce
Theorem~\ref{thm:alternate-inheritance} from the proposition and then
give two proofs of the proposition.

\begin{proof}[Proof of Theorem~\ref{thm:alternate-inheritance}]
Lemma~\ref{lem:density} gives~\eqref{eq:alternate-host-density}.
For~(i), apply Proposition~\ref{prop:alternate-sparse-pairs} to both
$(S,S')$ and $(S',S)$. Its conclusion holds in both directions except
with probability at most $2e^{-cm}$. Conditional on $v\in S$, the set
$S'$ is a uniform $m$-subset of $V(G)\setminus\{v\}$, and
$\mathbb E(|N^\sigma(v)\cap S'|\mid v\in S)\geq\eta m$.
Lemma~\ref{lem:chernoff-hoeffding} and a union bound show that, except
with probability at most $4m e^{-\eta^2m/2}$, all vertices of $S$ have
at least $\frac{\eta m}{2}$ inneighbors and outneighbors in $S'$, and
conversely. Here the factor $4m$ follows by summing the conditional
failure probabilities over the vertices in the two sampled sets.

Assume that neither exceptional event occurs, and let $A\subseteq S$
satisfy $0<|A|\leq\frac{m}{2}$. If $|A|\leq\frac{\eta m}{4}$, the
degree bounds give $|N^\pm(A)\cap S'|\geq\frac{\eta m}{2}\geq2|A|$.
Otherwise, suppose that out-expansion fails and set
$B=S'\setminus N^+(A)$. We have $e_G(A,B)=0$ and
$$
 |A|>\frac{\eta m}{4}\geq2\tau m,~
 |B|>\left(\frac{1}{2}-\frac{\nu}{8}\right)m,~
 |A|+|B|>\left(1-\frac{\nu}{8}\right)m.
$$
Since $\tau<\frac{1}{8}$, these inequalities
imply~\eqref{eq:alternate-sample-sizes}, a contradiction. For
in-expansion use $B=S'\setminus N^-(A)$ and the conclusion for $(S',S)$,
applied to $(B,A)$. Interchanging $S,S'$ gives the remaining two
conditions. Since $\eta\geq8\nu$, increasing $m_0(\nu)$ and decreasing
$c(\nu)$ absorbs $2e^{-cm}+4m e^{-\eta^2m/2}$ into the claimed bound.

For~(ii), suppose that $G[W]$ is not a robust
$(\frac{\nu^2}{64},2\tau)$-outexpander. Choose a witnessing set $A$ with
$2\tau m\leq|A|\leq(1-2\tau)m$, and set
$B=W\setminus\RN_{\nu^2/64,G[W]}(A)$. We have
$$
 |B|>m-|A|-\frac{\nu^2m}{64},~
 |A|+|B|>\left(1-\frac{\nu^2}{64}\right)m,~
 e_G(A,B)<\frac{\nu^2m^2}{64}.
$$
Since $\nu\leq\tau<\frac{1}{2}$, these sets
satisfy~\eqref{eq:alternate-sample-sizes}. Part~(ii) follows from
Proposition~\ref{prop:alternate-sparse-pairs}. For the degree estimate,
Lemma~\ref{lem:chernoff-hoeffding} gives
$\Pr(|N^\sigma(v)\cap W|<\frac{\eta m}{2})\leq e^{-\eta^2m/2}$ for
all $v\in V(G)$ and $\sigma\in\{+,-\}$; take a union bound over
the $2n$ choices.
\end{proof}

Lemma~\ref{lem:chernoff-hoeffding} gives Hoeffding's inequality for
sampling without replacement. The partition-testing proof also uses
the following version for sums of independent random variables.

\begin{lemma}[{Hoeffding~\cite[Theorem~1]{Hoeffding}}]
\label{lem:alternate-hoeffding}
Let $X_1,\ldots,X_r$ be independent random variables taking values in
$[0,1]$, and set $X=\sum_{i=1}^r X_i$. For all $t>0$,
$$
 \Pr(X-\mathbb E X\geq t)\leq e^{-2t^2/r}
 ~\text{and}~
 \Pr(X-\mathbb E X\leq-t)\leq e^{-2t^2/r}.
$$
\end{lemma}

Both proofs use Azuma's inequality in the following form.

\begin{lemma}[Azuma~\cite{Azuma}]
\label{lem:alternate-azuma}
Let $M_0,\ldots,M_r$ be a martingale with $M_0$ constant; thus
$\mathbb E(M_i\mid M_0,\ldots,M_{i-1})=M_{i-1}$ for all $i\in[r]$.
If
$|M_i-M_{i-1}|\leq a_i$ almost surely for all $i\in[r]$, where the
$a_i>0$ are constants, then, for all $t>0$,
$$
 \Pr(M_r-M_0\geq t)\leq
 \exp\left(-\frac{t^2}{2\sum_{i=1}^r a_i^2}\right)
 ~\text{and}~
 \Pr(M_r-M_0\leq-t)\leq
 \exp\left(-\frac{t^2}{2\sum_{i=1}^r a_i^2}\right).
$$
\end{lemma}

In both applications of Lemma~\ref{lem:alternate-azuma}, we expose
vertices successively in a uniform random ordering and take $M_i$ to
be the conditional expectation of the quantity being estimated after
$i$ exposures. If transposing two entries changes that quantity by
at most $a$, coupling two possible exposures by a transposition gives
$|M_i-M_{i-1}|\leq a$.

\subsection{Weak regularity and cut-norm sampling}
\label{app:weak-regularity}

The proof follows the argument of Dragani\'c, Keevash, and
M\"uyesser~\cite[Lemma~2.3]{DKMcyclic}: take a weak-regularity partition
and enlarge a sparse pair in the sample within each cell to obtain a
sparse pair in the original graph. A related argument for a single
sparse set appears in Sun, Wei, and Yang~\cite[Section~3]{SWY}; see also
Lang~\cite{Lang} for a framework based on families approximately closed
under subsampling. To allow $m$ to grow arbitrarily slowly, we use a
sampling theorem to bound the approximation error by a multiple of
$m^2$, rather than $n^2$.

For a real matrix $R$ indexed by a finite nonempty set $T$, its
\emph{normalized cut norm} is
$$
 \|R\|_\square=
 \frac{1}{|T|^2}\max_{X,Y\subseteq T}
 \left|\sum_{x\in X,\,y\in Y}R_{xy}\right|.
$$
The sets $X,Y$ may overlap. For a nonempty subset $Z\subseteq T$,
$R[Z]$ denotes the matrix restricted to rows and columns in $Z$;
its cut norm is normalized by $|Z|^2$.

A partition $\mathcal P=\{V_1,\ldots,V_k\}$ into nonempty classes is
\emph{equitable} if their sizes differ by at most one. For a
digraph $G$ with adjacency matrix $M$ and such a partition, set
$$
 d_{ij}=\frac{e_G(V_i,V_j)}{|V_i||V_j|},~
 (M_{\mathcal P})_{xy}=d_{ij}
 ~\text{if}~x\in V_i,~ y\in V_j.
$$
Thus $M_{\mathcal P}$ replaces each block of $M$ by its average.
For $X,Y\subseteq V(G)$, write
$$
 e_{\mathcal P}(X,Y)=
 \sum_{i,j\in[k]}d_{ij}|X\cap V_i||Y\cap V_j|.
$$
The definition includes $i=j$ and does not require $d_{ij}=d_{ji}$.

We use the following directed form of the Frieze--Kannan weak
regularity lemma~\cite[Theorem~2 and Section~5.1]{FK}. The equitable
graph version is stated in~\cite[Lemma~3.3]{DKMcyclic}; the matrix
argument applies without a symmetry assumption.

\begin{lemma}[Frieze--Kannan]\label{lem:alternate-fk}
For all $0<\delta<1$, there exist $K=K(\delta)$ and $n_0=n_0(\delta)$
such that all digraphs $G$ of order $n\geq n_0$ have an equitable
partition $\mathcal P=\{V_1,\ldots,V_k\}$ into $k\leq K$ nonempty
classes satisfying $\|M-M_{\mathcal P}\|_\square\leq\delta$.
Equivalently,
\begin{equation}\label{eq:alternate-fk-error}
 |e_G(X,Y)-e_{\mathcal P}(X,Y)|\leq\delta n^2
 ~\text{for all}~X,Y\subseteq V(G).
\end{equation}
\end{lemma}

The cut-norm sampling theorem of Alon, Fernandez de la Vega, Kannan,
and Karpinski gives the following bounded-matrix consequence, stated
using our normalization.\footnote{Theorem~9 refers to the
\href{https://web.math.princeton.edu/~nalon/PDFS/adkk8.pdf}{conference manuscript}
of~\cite{AFKK}.}

\begin{theorem}[{Alon et al.~\cite[Theorem~9]{AFKK}}]
\label{thm:alternate-cut-sampling}
For all $0<\varepsilon<1$, there exists $s_0=s_0(\varepsilon)$ such
that the following holds. If $R\in[-1,1]^{T\times T}$ and $Z$ is a
uniformly random $s$-subset of $T$, where $s_0\leq s\leq|T|$, then
$$
 \Pr\bigl(\|R[Z]\|_\square>
          \|R\|_\square+\varepsilon\bigr)\leq\frac{1}{40}.
$$
No symmetry assumption is imposed on $R$.
\end{theorem}

To obtain this statement from the cited theorem, take the array
dimension to be $2$ and the entry bound to be $1$. Since
$\sum_{x,y\in T}R_{xy}^2\leq|T|^2$, its normalized error terms are
at most $10a^2+5a$ when $s\geq\frac{128000}{a^6}$. Taking $a$ small
enough in terms of $\varepsilon$ gives the statement.

The exponential bound needed below follows from this theorem and
Lemma~\ref{lem:alternate-azuma}.

\begin{lemma}\label{lem:alternate-cut-sampling}
For all $0<\varepsilon<1$, there exists $s_1=s_1(\varepsilon)$ such
that, for all $R\in[-1,1]^{T\times T}$ and uniformly random
$s$-subsets $Z$ with $s_1\leq s\leq|T|$,
\begin{equation*}
 \Pr\bigl(\|R[Z]\|_\square>
          \|R\|_\square+\varepsilon\bigr)
 \leq e^{-\varepsilon^2s/128}.
\end{equation*}
\end{lemma}

\begin{proof}
Replacing one vertex of $Z$ changes only one row and one column of
$R[Z]$, so it changes $f(Z)=\|R[Z]\|_\square$ by at most
$\frac{4}{s}$. Exposing the $s$ sampled vertices and applying
Lemma~\ref{lem:alternate-azuma} gives
$\Pr(f-\mathbb E f\geq t)\leq e^{-t^2s/32}$ and
$\Pr(f-\mathbb E f\leq-t)\leq e^{-t^2s/32}$ for all $t>0$.
Theorem~\ref{thm:alternate-cut-sampling}, with
$\frac{\varepsilon}{4}$ in place of $\varepsilon$, gives
$\Pr(f\leq\|R\|_\square+\frac{\varepsilon}{4})\geq\frac{39}{40}$.
If $\mathbb E f>\|R\|_\square+\frac{\varepsilon}{2}$, the lower-tail
bound would make this probability at most $e^{-\varepsilon^2s/512}$,
a contradiction for sufficiently large $s$. Hence
$\mathbb E f\leq\|R\|_\square+\frac{\varepsilon}{2}$, and the upper-tail
bound with $t=\frac{\varepsilon}{2}$ proves the lemma.
\end{proof}

\begin{proof}[First proof of Proposition~\ref{prop:alternate-sparse-pairs}]
Set $\delta=\frac{\nu^2}{1024}$, and take $\mathcal P$ from
Lemma~\ref{lem:alternate-fk}. Let $m_0$ be sufficiently
large in terms of $\nu$, so in particular $k\leq\frac{\nu n}{32}$,
and set $p=\frac{m}{n}$. Equitability gives
$p|V_i|\geq\frac{m}{2k}$. By Lemma~\ref{lem:chernoff-hoeffding} and a
union bound,
\begin{equation}\label{eq:alternate-fk-cell-sample}
 |T\cap V_i|=\left(1\pm\frac{\nu}{8}\right)p|V_i|
 ~\text{for all}~i\in[k]~\text{and}~T\in\{U,V\},
\end{equation}
except with probability at most
$4k\exp(-\frac{\nu^2m}{128k^2})$.

Set $R=M-M_{\mathcal P}$, where $M$ is the adjacency matrix of $G$.
We have $R_{xy}\in[-1,1]$ and
$\|R\|_\square\leq\delta$. The set $Z=U\cup V$ is uniform of size
$m$ or $2m$, according to the sampling model.
By Lemma~\ref{lem:alternate-cut-sampling}, with probability at least
$1-e^{-\delta^2m/128}$,
\begin{equation}\label{eq:alternate-fk-sampled-error}
 |e_G(A,B)-e_{\mathcal P}(A,B)|\leq8\delta m^2
 ~\text{for all}~A\subseteq U,~ B\subseteq V.
\end{equation}
Here we used $\|R[Z]\|_\square\leq2\delta$ and $|Z|\leq2m$.

Assume that~\eqref{eq:alternate-fk-cell-sample}
and~\eqref{eq:alternate-fk-sampled-error} hold. Suppose for a
contradiction that $A\subseteq U$, $B\subseteq V$
satisfy~\eqref{eq:alternate-sample-sizes} and
$e_G(A,B)\leq\frac{\nu^2m^2}{64}$. Choose sets
$P,Q\subseteq V(G)$ so that
$$
 |P\cap V_i|=\left\lfloor\frac{|A\cap V_i|\,|V_i|}{|U\cap V_i|}\right\rfloor
 ~\text{and}~
 |Q\cap V_i|=\left\lfloor\frac{|B\cap V_i|\,|V_i|}{|V\cap V_i|}\right\rfloor
$$
for all $i\in[k]$. The sets $P,Q$ may overlap. From the cell counts,
the sample-size conditions, and $k\leq\frac{\nu n}{32}$, we obtain
\begin{align*}
 |P|,|Q|\geq\frac{3\tau n}{2(1+\nu/8)}-k\geq\tau n~\text{and}~|P|+|Q|\geq\frac{1-\nu/4}{1+\nu/8}n-2k
          \geq\left(1-\frac{\nu}{2}\right)n.
\end{align*}
Also, all $d_{ij}$ are nonnegative, so
$e_{\mathcal P}(P,Q)\leq\frac{e_{\mathcal P}(A,B)}{(1-\nu/8)^2p^2}$.
Combining this inequality with~\eqref{eq:alternate-fk-error}
and~\eqref{eq:alternate-fk-sampled-error} gives
$$
 e_G(P,Q)\leq
 \left(\delta+\frac{\nu^2/64+8\delta}{(1-\nu/8)^2}\right)n^2
 <\frac{\nu^2n^2}{4},
$$
contrary to~\eqref{eq:alternate-host-density}. Since $k$ is bounded in
terms of $\nu$, a union bound on the two exceptional events proves
the proposition after decreasing $c$ and increasing $m_0$.
\end{proof}

\subsection{Partition testing}

The second proof uses the partition-testing theorem of Fischer,
Matsliah, and Shapira. We first give their definitions and theorem in
the special case needed here.\footnote{Numbered references to~\cite{FMS}
in this subsection follow the
\href{https://www.math.tau.ac.il/~asafico/regalg.pdf}{authors' manuscript}.}

A \emph{marked digraph} is a digraph $H$ together with a specified
subset $L\subseteq V(H)$. Let $N=|V(H)|$. For an ordered partition
$\mathcal C=(C_1,\ldots,C_k)$ of $V(H)$, allowing empty classes, its
\emph{density vector} consists of the coordinates
$$
 \rho_i=\frac{|C_i|}{N},~
 \mu_{ij}=\frac{e_H(C_i,C_j)}{N^2},~
 \lambda_i=\frac{|C_i\cap L|}{N}
 ~(\text{for all } i,j\in[k]).
$$
The $\lambda_i$ are omitted for an unmarked digraph. All ordered
pairs $(i,j)$ are included, also when $i=j$; the edge counts are
normalized by $N^2$, not by $|C_i||C_j|$. These are the density
tensors of~\cite[Definition~2.1]{FMS} for one binary relation
(directed edges), with one unary relation (the marking) when needed.

Let $\Psi$ be a set of allowed vectors with coordinates in $[0,1]$.
Following~\cite[Definition~2.2]{FMS}, a partition \emph{satisfies}
$\Psi$ if its density vector belongs to $\Psi$. It
\emph{$\varepsilon$-approximately satisfies} $\Psi$ if there exists
$\psi\in\Psi$ differing from its density vector by at most
$\varepsilon$ in all coordinates. It \emph{$\varepsilon$-closely
satisfies} $\Psi$ if, in addition, its class-size coordinates
$\rho_i$ agree exactly with those of $\psi$. The digraph $H$
satisfies any of these conditions if there exists a partition of
$V(H)$ satisfying it.

A \emph{tester} is a randomized algorithm which is given $N$ and
access to the vertices labelled by $[N]$. It may sample uniform
vertices, ask whether a specified ordered pair is an edge, and,
in the marked case, ask whether a specified vertex belongs to $L$.
Its output is either accept or reject. The bound below concerns
the number of vertices it examines, independently of $N$.

\begin{theorem}[{Fischer--Matsliah--Shapira~\cite[Theorem~3]{FMS}}]
\label{thm:alternate-partition-testing}
Fix $k\in\mathbb N$ and $0<\varepsilon<1$. Let $\Psi$ be a set of
density vectors specified by a fixed finite system of linear
equalities and weak inequalities. There exists a tester using at
most $q=q(k,\varepsilon)$ vertices in its samples and queries such
that, for all marked or unmarked digraphs $H$,
\begin{enumerate}[label=\textup{(\roman*)}]
\item if $H$ satisfies $\Psi$, it accepts with probability at least
$\frac{15}{16}$;
\item if $H$ does not $\varepsilon$-closely satisfy $\Psi$, it rejects
with probability at least $\frac{15}{16}$.
\end{enumerate}
Here $H$ has the marking convention used in defining $\Psi$.
The bound $q$ is independent of $|V(H)|$ and of the coefficients
in the constraints defining $\Psi$.
\end{theorem}

The source theorem allows more general sets $\Psi$, provided one
can check whether a given vector is within a prescribed distance
in all coordinates of a vector in $\Psi$, and produce such a vector
when one exists~\cite[Definition~2.3]{FMS}. For a fixed system of
linear constraints, this is a linear feasibility problem: intersect
the constraints defining $\Psi$ with the prescribed interval in
each coordinate. Thus the stated specialization meets that
hypothesis. The proof of~\cite[Theorem~3]{FMS} uses sampling without
a regularity lemma. We will use the weaker consequence of~(ii)
that the tester rejects with probability at least $\frac{15}{16}$
whenever $H$ does not $\varepsilon$-approximately satisfy $\Psi$.

\begin{proof}[Second proof of Proposition~\ref{prop:alternate-sparse-pairs}]
Since the partition tester is given the order of its input, we compare
balanced blow-ups of the host and the sample having the same order.
After uniform random labelling, a tester examining only a bounded number
of vertices has nearly the same acceptance probability on these two inputs.
This will bound away from one the probability of a pair approximately
satisfying the size and sparsity conditions; Azuma's inequality will then
give the required exponential bound.
We first establish a weighted consequence of~\eqref{eq:alternate-host-density},
which ensures that the host blow-ups exclude the relevant sparse pairs.
For all sufficiently large $n$ in terms of $\nu$, if
$a_v,b_v\in[0,1]$ satisfy
$$
 \sum_v a_v,~ \sum_v b_v\geq\frac{5\tau n}{4},~\text{and}~
 \sum_v(a_v+b_v)\geq\left(1-\frac{3\nu}{8}\right)n,
$$
then
\begin{equation}\label{eq:alternate-weighted-density}
 \sum_{uv\in E(G)}a_u b_v\geq\frac{\nu^2n^2}{8}.
\end{equation}
Indeed, choose random sets $P,Q$ by including $v$ with probabilities
$a_v,b_v$, independently over all $2n$ choices.
Lemma~\ref{lem:alternate-hoeffding} shows that the size conditions
in~\eqref{eq:alternate-host-density} hold with probability at least
$1-3e^{-\nu^2n/64}\geq\frac{1}{2}$. Taking the expectation of
$e_G(P,Q)$ proves~\eqref{eq:alternate-weighted-density}.

For a digraph $H$, let $H^{(b)}$ denote its balanced blow-up, in which
each vertex is replaced by $b$ clones and each edge by all edges
between the corresponding clone classes. Marked vertices have
marked clones. Set $N=bn$. Applying~\eqref{eq:alternate-weighted-density}
to the proportions of clones in two sets shows that, for all
$A,B\subseteq V(G^{(b)})$ with $|A|,|B|\geq\frac{5\tau N}{4}$ and
$|A|+|B|\geq(1-\frac{3\nu}{8})N$, we have
\begin{equation}\label{eq:alternate-blowup-density}
 e(A,B)\geq\frac{\nu^2N^2}{8}.
\end{equation}

Set $\varepsilon=\frac{\nu^2}{1024}$. To measure the simultaneous
failure of the inequalities required for a sparse pair, define
\begin{equation*}
 F(U,V)=\min_{\substack{A\subseteq U\\B\subseteq V}}\max\Bigl\{
 \frac{3\tau}{2}-\frac{|A|}{m},~
 \frac{3\tau}{2}-\frac{|B|}{m},~
 1-\frac{\nu}{4}-\frac{|A|+|B|}{m},~
 \frac{e_G(A,B)}{m^2}-\frac{\nu^2}{64},~0\Bigr\}.
\end{equation*}
Thus $F(U,V)=0$ precisely when the conclusion of the proposition
fails. We show first that $\Pr(F(U,V)\leq\varepsilon)<\frac{1}{4}$.

For a digraph $H$ of order $N$, consider the property that there
exist sets $A,B$ satisfying
\begin{equation}\label{eq:alternate-relaxed-pair}
 |A|,|B|\geq\left(\frac{3\tau}{2}-\varepsilon\right)N,~
 |A|+|B|\geq\left(1-\frac{\nu}{4}-\varepsilon\right)N
 \text{,~and~}e_H(A,B)\leq\left(\frac{\nu^2}{64}+\varepsilon\right)N^2.
\end{equation}
Using the four classes $C_1=A\setminus B$, $C_2=B\setminus A$,
$C_3=A\cap B$, and $C_4=V(H)\setminus(A\cup B)$, these become
the linear constraints
$$
 \rho_1+\rho_3,~\rho_2+\rho_3\geq\frac{3\tau}{2}-\varepsilon,~
 \rho_1+\rho_2+2\rho_3\geq1-\frac{\nu}{4}-\varepsilon
 \text{,~and~}\mu_{12}+\mu_{13}+\mu_{32}+\mu_{33}
 \leq\frac{\nu^2}{64}+\varepsilon.
$$
Conversely, a partition satisfying these constraints gives the
required sets $A=C_1\cup C_3$ and $B=C_2\cup C_3$.

For the disjoint sampling model, let $B_G$ be the bipartite digraph
on left and right copies of $V(G)$, with edges $u_Lv_R$ for all
$uv\in E(G)$ and with the left side marked. Let $B_{U,V}$ be the
corresponding digraph on copies of $U,V$. On a marked bipartite
digraph with sides $L,R$ of order $N$, require~\eqref{eq:alternate-relaxed-pair}
with $A\subseteq L$, $B\subseteq R$, using the classes
$A,L\setminus A,B,R\setminus B$. Since the total order is $2N$,
the normalized constraints are
$$
 \rho_1,~\rho_3\geq\frac{3\tau/2-\varepsilon}{2},~
 \rho_1+\rho_3\geq\frac{1-\nu/4-\varepsilon}{2}
 \text{,~and~}\mu_{13}\leq\frac{\nu^2/64+\varepsilon}{4},
$$
together with $\lambda_1=\rho_1$, $\lambda_2=\rho_2$, and
$\lambda_3=\lambda_4=0$. The marking equalities place the first two classes in $L$ and the
last two in $R$.

No blow-up of $G$ or $B_G$ $\varepsilon$-approximately satisfies
the corresponding set of allowed vectors. In the first case, such a
partition would give sets satisfying~\eqref{eq:alternate-relaxed-pair}
with an additional error of at most $2\varepsilon N$ in either
individual size, $4\varepsilon N$ in the sum, and $4\varepsilon N^2$
in the edge count. In the marked case, let $C_1,\ldots,C_4$ be the
partition. Approximate satisfaction of the marking constraints gives
$|C_1\setminus L|\leq4\varepsilon N$ and
$|C_3\cap L|\leq2\varepsilon N$. Set $A=C_1\cap L$ and
$B=C_3\cap R$. The additional errors are at most $6\varepsilon N$
in either size, $10\varepsilon N$ in the sum, and
$4\varepsilon N^2$ in the edge count. In both cases, the weaker bounds
$$
\begin{gathered}
 |A|,|B|\geq\left(\frac{3\tau}{2}-17\varepsilon\right)N
 \geq\frac{5\tau N}{4},~
 |A|+|B|\geq\left(1-\frac{\nu}{4}-17\varepsilon\right)N
 \geq\left(1-\frac{3\nu}{8}\right)N\text{,~and}\\
 e(A,B)\leq\left(\frac{\nu^2}{64}+17\varepsilon\right)N^2
 <\frac{\nu^2N^2}{8}
\end{gathered}
$$
contradict~\eqref{eq:alternate-blowup-density}. For $B_G^{(b)}$, identify
its left and right copies with $V(G^{(b)})$; the resulting sets may
overlap, and their directed edge count is unchanged.

Apply Theorem~\ref{thm:alternate-partition-testing} with accuracy
$\varepsilon$ to these two properties. Let $q=q(\nu)$ bound the
number of vertices used by either tester. Compare the tester
on $G^{(m)}$ and $(G[U])^{(n)}$ when $U=V$, and on $B_G^{(m)}$ and
$B_{U,V}^{(n)}$ otherwise. The two inputs have the same order and,
in the marked case, the same number of marked vertices. Give the
inputs uniform random labellings.

Their acceptance probabilities differ by at most
$\frac{8q^2}{m}$, where the probability for the second input includes
the random sample. To see this, expose the distinct labels used by
the tester, coupling their marks in the marked case. A new label is a
uniform unused vertex, even if the query was chosen using earlier
answers. After $j<q$ exposures, the probabilities of projecting onto
an already exposed vertex of $G$ are at most $\frac{j}{n-q}$ and
$\frac{j}{m-q}$, respectively. In the marked host graph, a projection
onto a vertex already seen on the other side also counts as a
repetition. Conditional on no repetition, the projected vertices
are uniform distinct vertices of $G$ in both experiments, so all
answers can be coupled to agree. A union bound gives the asserted
estimate when $m\geq2q$.

Take $m_0$ large enough that $m\geq m_0$ implies $m\geq128q^2$.
The tester accepts the host blow-up with probability at most
$\frac{1}{16}$, and hence the sampled blow-up with probability at
most $\frac{1}{8}$. If $F(U,V)\leq\varepsilon$, the sample and all
its balanced blow-ups satisfy the appropriate property exactly;
the conditional acceptance probability is at least $\frac{15}{16}$.
It follows that
$\Pr(F(U,V)\leq\varepsilon)\leq\frac{2}{15}<\frac{1}{4}$, and
$\mathbb E F(U,V)\geq\frac{3\varepsilon}{4}$.

Finally, realize the sample as the first $m$ vertices of a uniform
random ordering, or as its first two blocks of $m$ vertices.
Transposing two entries changes $F$ by at most $\frac{2}{m}$:
transfer the memberships in $A,B$ to the replacement vertices,
preserving the sizes and changing at most $2m$ directed edges.
Exposing at most $2m$ entries and applying
Lemma~\ref{lem:alternate-azuma} therefore gives
\begin{equation*}
 \Pr(F(U,V)=0)
 \leq\exp\left(-\frac{(\mathbb E F(U,V))^2m}{16}\right)
 \leq\exp\left(-\frac{9\varepsilon^2m}{256}\right).\qedhere
\end{equation*}
\end{proof}

\end{document}